\documentclass[oneside,english]{amsart}
\usepackage[T1]{fontenc}
\usepackage[utf8]{luainputenc}
\usepackage{color}
\usepackage{babel}
\usepackage{refstyle}
\usepackage{amsthm}
\usepackage{amssymb}

\usepackage[all]{xy}
\usepackage[unicode=true,pdfusetitle,
 bookmarks=true,bookmarksnumbered=false,bookmarksopen=false,
 breaklinks=false,pdfborder={0 0 1},backref=page,colorlinks=true]
 {hyperref}

\makeatletter

\AtBeginDocument{}
\AtBeginDocument{}
\AtBeginDocument{}
\AtBeginDocument{\providecommand\corref[1]{\ref{cor:#1}}}
\AtBeginDocument{}
\AtBeginDocument{\providecommand\eqref[1]{\ref{eq:#1}}}
\RS@ifundefined{subref}
  {\def\RSsubtxt{section~}\newref{sub}{name = \RSsubtxt}}
  {}
\RS@ifundefined{thmref}
  {\def\RSthmtxt{theorem~}\newref{thm}{name = \RSthmtxt}}
  {}
\RS@ifundefined{lemref}
  {\def\RSlemtxt{lemma~}\newref{lem}{name = \RSlemtxt}}
  {}

\numberwithin{equation}{section}
\numberwithin{figure}{section}
  \theoremstyle{plain}
  \newtheorem*{thm*}{\protect\theoremname}
\theoremstyle{plain}
\newtheorem{thm}{\protect\theoremname}[section]
  \theoremstyle{plain}
  \newtheorem{prop}[thm]{\protect\propositionname}
  \theoremstyle{definition}
  \newtheorem{defn}[thm]{\protect\definitionname}
  \theoremstyle{plain}
  \newtheorem{lem}[thm]{\protect\lemmaname}
  \theoremstyle{plain}
  \newtheorem{cor}[thm]{\protect\corollaryname}
  \theoremstyle{remark}
  \newtheorem{rem}[thm]{\protect\remarkname}
  \newtheorem{example}[thm]{\protect\examplename}

\makeatother

  \providecommand{\corollaryname}{Corollary}
  \providecommand{\definitionname}{Definition}
  \providecommand{\lemmaname}{Lemma}
  \providecommand{\claimname}{Claim}
  \providecommand{\examplename}{Example}
  \providecommand{\propositionname}{Proposition}
  \providecommand{\remarkname}{Remark}
  \providecommand{\theoremname}{Theorem}
\providecommand{\theoremname}{Theorem}
\providecommand{\notename}{Note}
\providecommand{\acknowledgmentname}{Acknowledgment}
\newtheorem{convention}[thm]{Convention}
\begin{document}

\title{On generic $G$-graded Azumaya algebras}

\author{Eli Aljadeff \and Yaakov karasik}

\address{Department of Mathematics, Technion - Israel Institute of Technology,
Haifa 32000, Israel}

\address{Institut fur Algebra und Geometrie, KIT, Karlsruhe, Germany}

\email{aljadeff 'at' tx.technion.ac.il (E. Aljadeff),}

\email{igor.karasik 'at' kit.edu (Y. Karasik).}

\thanks{The first author was supported by the ISRAEL SCIENCE FOUNDATION (grant No 1516/16)}

\begin{abstract}

Let $F$ be an algebraically closed field of characteristic zero and let $G$ be a finite group. Consider $G$-graded simple algebras $A$ which are finite dimensional and $e$-central over $F$, i.e. $Z(A)_{e} := Z(A)\cap A_{e} = F$. For any such algebra we construct a \textit{generic} $G$-graded algebra $\mathcal{U}$ which is \textit{Azumaya} in the following sense. $(1)$ \textit{$($Correspondence of ideals$)$}: There is one to one correspondence between the $G$-graded ideals of  $\mathcal{U}$ and the ideals of the ring $R$, the $e$-center of $\mathcal{U}$. $(2)$ \textit{Artin-Procesi condition}: $\mathcal{U}$ satisfies the $G$-graded identities of $A$ and no nonzero $G$-graded homomorphic image of  $\mathcal{U}$ satisfies properly more identities. $(3)$ \textit{Generic}: If $B$ is a $G$-graded algebra over a field then it is a specialization of $\mathcal{U}$ along an ideal $\mathfrak{a} \in spec(Z(\mathcal{U})_{e})$ if and only if it is a $G$-graded form of $A$ over its $e$-center.

We apply this to characterize finite dimensional $G$-graded simple algebras over $F$ that admit a $G$-graded division algebra form over their $e$-center.

\end{abstract}

\maketitle

\section{introduction}

Group gradings on finite dimensional algebras over a field $k$ and in particular on $k$-central simple algebras is key in the study of Brauer groups and division algebras. Indeed, the indispensable interpretation of $Br(k)$ in terms of Galois cohomology may be realized via a group gradings on $k$-central simple algebras. These are the well known \textit{crossed product} algebras. A rather different group grading, \textit{finer} than the crossed product, is the well known \textit{symbol algebra} construction which determines a group grading on $k$-central simple algebras with a group $H$ isomorphic to $\mathbb{Z}_{n} \times \mathbb{Z}_{n}$, where $\mathbb{Z}_{n}$ denotes the cyclic group of order $n$. This grading establishes the fundamental connection between $Br(k)$ and the $K$-group $K_{2}$. It will not be an exaggeration if we say that several of the main theorems and open questions in the theory of Brauer groups and division algebras involve the existence of a group grading on $k$-central simple algebras. To mention a few $(1)$ Merkurjev-Suslin's theorem on Brauer classes represented by products of symbol algebras $(2)$ Is every Brauer class represented by a $G$-crossed product where $G$ is abelian? $(3)$ Is every class represented by a product of cyclic algebras? $(4)$ Is every division algebra of prime index cyclic? $(5)$ Amitsur's theorem on noncrossed product division algebras $(6)$ Every division algebra over a local or global field is cyclic. For the reader convenience we recall below both constructions, namely the crossed product and the symbol algebra.

The geometric point of view in this context, namely the generic division algebra of degree $n$, the algebra of generic matrices or the generic $G$-crossed product where $G$ is a finite group, has been of great importance as well in the study of Brauer groups and division algebras. Here again we can mention Amitsur's example of noncrossed product division algebra, rationality questions on the center of the degree $n$ generic division algebra and topic which was studied in the last $20$ years or so, the essential dimension central simple algebras and in particular of the generic division algebra of degree $n$. Here are some well known references with some of the statements$/$open problems mentioned above (\cite{MerSus}, \cite{Jacobson}, \cite{Amitsur noncrossed products}, \cite{SERRE},\cite{Brussel}, \cite{EssDim}, \cite{Vishne etal problems}).

Our goal in this paper is to construct a generic $G$-graded algebra which is \linebreak Azumaya. In order to explain our terminology let $G$ be any finite group and let $F$ an algebraically closed field of characteristic zero. We let $A$ be a $G$-graded simple algebra and finite dimensional over its $e$-center $F$. By definition this means that $A\cdot A \neq 0$, $A$ has no nontrivial $G$-graded ideals and is finite dimensional over $F$, where $F$ is the intersection of $Z(A)$ and $A_{e}$, the center of $A$ and its $e$-component respectively. The following is the main result of the paper.

\begin{thm}\label{main theorem} Let $A$ be a $G$-graded simple algebra as above where $G$ is finite. Then there exists a $G$-graded algebra $\mathcal{U} = \mathcal{U}_{A}$ over a commutative ring $R$, the $e$-center of $\mathcal{U}$, with the following properties:

\begin{enumerate}

\item
The ring $R$ is an integral domain.

\item
The algebra $\mathcal{U}$ satisfies all graded identities of $A$, that is $Id_{G, F}(\mathcal{U}) \supseteq Id_{G, F}(A)$.

\item
If $\eta$ is a nonzero $G$-graded morphism defined on $\mathcal{U}$ and $Im(\eta)$ is its image, then $Id_{G, F}(Im(\eta)) \subseteq Id_{G, F}(A)$ and hence equality holds. We say that no nonzero $G$-graded homomorphic image of $\mathcal{U}$ satisfies more identities than $A$.

\item

There is a $1-1$ correspondence between the set of $2$-sided $G$-graded ideals of $\mathcal{U}$ and ideals of $R$. Moreover, the correspondence is given by the following maps which are inverses of each other
$$
I \rightarrow I\mathcal{U}
$$
$$
T \rightarrow T\cap R
$$
where I is an ideal of $R$ and $T$ is a $2$-sided $G$-graded ideal of $\mathcal{U}$.
\item

The algebra $\mathcal{U}$ is generic $($or representing$)$. This means the following: If $B$ is a $G$-graded algebra over a field then it is a specialization of $\mathcal{U}$ along an ideal $\mathfrak{a} \in spec(Z(\mathcal{U})_{e})$ if and only if it is a $G$-graded form of $A$ over its $e$-center. More precisely, this means:

\begin{enumerate}
\item
If $p$ is a prime ideal of $R$ then $\mathcal{U}/p\mathcal{U}$ is a $G$-graded form of $A$, that is, if $E$ is an algebraically closed field extending $R/pR$ and $F$, then $E \otimes_{R/pR} \mathcal{U}/p\mathcal{U} \cong E\otimes_{F}A$ as $G$-graded algebras.

\item
Conversely: If $C$, an algebra over a field, is a $G$-graded form of $A$ over $Z(C)_{e}$, then there exists a specialization $\mathcal{U}/p\mathcal{U}$ of $\mathcal{U}$ which extends to $C$ as $G$-graded algebras.

\end{enumerate}
\end{enumerate}

\end{thm}

\begin{rem}
Note that if $G$ is the trivial group the theorem above just says that $\mathcal{U}$ is Azumaya over its center $R$. Indeed, this follows from the well known theorem of Artin and Procesi, which says that an algebra is Azumaya of rank $n$ over its center $R$ if and only if it satisfies the identities of $n \times n$-matrices and no nonzero homomorphic image of $\mathcal{U}$ satisfies the identities of $(n-1) \times (n-1)$-matrices.
\end{rem}

\begin{rem}
Note that we are not proving the generalization of Artin-Procesi's theorem for $G$-graded simple algebras. In particular we are not showing that any $G$-graded algebra which is Azumaya in the sense of Artin-Procesi has the $1-1$ correspondence property on ideals. We refer the reader to [\cite{AGPR}, Chapter $5$] for more on Azumaya algebras.
\end{rem}
In view of these remarks and the $2$nd and $3$rd statements in the theorem we make the following definition.

\begin{defn}
A $G$-graded algebra $\mathcal{U}$ is Azumaya in the sense of Artin-Procesi over its $e$-center $R$ if it satisfies the $G$-graded identities of a finite dimensional $G$-graded simple algebra and no nonzero $G$-graded homomorphic image satisfies properly more identities.
\end{defn}

Because we want the generic $G$-graded algebra $\mathcal{U}$ to specialize to \textit{every} $G$-graded form of $A$, whose $e$-center is an algebra over a field, it is important to have \textit{a unique minimal field of definition $k$ of $A$}. Furthermore, because the algebra $\mathcal{U}$ will be constructed via a suitable localization of a $G$-graded relatively free algebra $\mathcal{A} = k\langle X_{G}\rangle/Id_{G, k}(A)$, where $k\langle X_{G}\rangle$ is the free $G$-graded algebra over the field $k$ and $Id_{G, k}(A)$ is the $T$-ideal of $G$-graded identities of $A$ defined over $k$, we need to know the $T$-ideal $Id_{G, F}(A)$ is indeed defined over $k$. Here $X_{G} = \cup_{g\in G}X_{g}$ and $X_{g}$ is a countable set of variables indexed by $g\in G$. These two notions, namely fields of definition of $A$ and of $Id_{G, F}(A)$, are tightly related. It is easy to show that if $k$ is a field of definition of $A$ it is also a field of definition of $Id_{G, F}(A)$. A substantial part of the paper is devoted to showing the existence and an explicit construction of the unique minimal field of definition $k$ of the $G$-graded algebra $A$. We do this in two steps: Firstly, we construct explicitly a certain field $k$ which turns out to be a field of definition of $A$. Secondly, we show it is unique minimal by showing that if a field $L$ does not contain $k$ then there is an identity of $A$ which is\textit{ not} defined over $L$.

For future reference we record this in the following theorem.

\begin{thm}\label{field of definition theorem}
Let $A$ be a finite dimensional $G$-graded simple algebra as above. The following hold.

\begin{enumerate}

\item
There exists a unique minimal subfield $k$ of $F$ over which $A$ is defined, that is there is a $G$-graded simple algebra $A_{0}$, finite dimensional and $e$-central over $k$, such that $A_{0}\otimes_{k}F \cong A$ as $G$-graded algebras and if $L$ is a subfield of $F$ over which $A$ is defined then $L \geq k$. Furthermore, the field $k$ is contained in a finite cyclotomic extension of $\mathbb{Q}$.

\item
The field $k$ is also the unique minimal field of definition of the $T$-ideal of $G$-graded identities $Id_{G, F}(A)$.

\end{enumerate}
\end{thm}
\begin{rem} \label{Ehud Meir Example}
In \cite{Ehud} Section $11$ there is an example of a family of $3$-dimensional nilpotent associative algebras $A_{a}$, parametrized by a parameter $a$. The algebra $A_a$ has then $\mathbb{Q}(a)$ as a unique minimal field of definition, while all the polynomial identities are generated by XYZ=0 and are therefore defined over $\mathbb{Q}$. The element $a$ may be any element $\neq 0,1$ of any extension field of $\mathbb{Q}$, and therefore in this case the minimal field of definition may be larger than the field of definition of the T-ideal of identities.
\end{rem}

In the Section $6$ we use our construction in order to characterize finite dimensional $G$-graded simple algebras $A$ that admit a graded division algebra form. These algebras can be characterized in terms of their graded presentation $P_{A}$ as given by Bahturin et. al. in \cite {BSZ} (or Section $2$ below) and also in terms of their $G$-graded $T$-ideal of identities. Let us explain our characterization in terms of identities and for that we recall Kemer's theory on verbally prime $T$-ideals. By definition a $T$-ideal $\Gamma$ is verbally prime if for any two $T$-ideals $I$ and $J$ such that $IJ \subseteq \Gamma$ then $I \subseteq \Gamma$ or $J \subseteq \Gamma$. Verbally prime $T$-ideals can be characterised also by means of polynomials, that is, $\Gamma$ is verbally prime if and only if for any two multilinear polynomials $p$ and $q$ with disjoint sets of variables such that $pq \in \Gamma$ then either $p \in \Gamma$ or $q \in \Gamma$. Kemer's theorem gives a concrete characterization of verbally prime $T$-ideals in terms of algebras $A$ where $\Gamma = Id(A)$ (see \cite{Kemer1984}). The theorem (which is well known) says that in case $\Gamma$ contains a Capelli polynomial then $\Gamma$ is verbally prime if and only if it is the $T$-ideal of identities of a matrix algebra $M_{n}(F)$, some $n$. In case $\Gamma$ contains no Capelli polynomial then $\Gamma$ is verbally prime if and only if it is the $T$-ideal of identities of $E(A)$, the Grassmann envelope of $A$, where $A$ is a finite dimensional $\mathbb{Z}_{2}$-graded simple algebra. Both definitions of verbally prime, namely in terms of $T$-ideals or polynomials, can be naturally extended to the $G$-graded setting where $G$ is a finite group. However, interestingly, the notions are not longer equivalent. It turns out that the definition in terms of polynomials is considerably stronger than the definition in terms of ideals. Consequently, the authors in \cite{AljKarasik} referred to the notion in terms of polynomials as strongly verbally prime. Here are the definitions.
\begin{defn}\label{definition verbally and strongly verbally}
Let $\Gamma$ be a $G$-graded $T$ ideal.
\begin{enumerate}
\item
We say $\Gamma$ is $G$-graded verbally prime if for any $G$-graded $T$-ideals $I$ and $J$, $IJ \subseteq \Gamma \Rightarrow$ either $I \subseteq \Gamma$ or $J \subseteq \Gamma$.
\item
We say $\Gamma$ is $G$-graded strongly verbally prime if for any $G$-homogeneous, multilinear polynomials $p$ and $q$ on disjoint sets of variables, $pq \in \Gamma \Rightarrow$ $p \in \Gamma$ or $q \in \Gamma$.
\end{enumerate}
\end{defn}
For $G$-graded verbally prime $T$-ideals the theory is similar to Kemer's. See \cite{AljKarasik}, Theorem $1.6$ for a complete characterization.
In case $\Gamma$ contains a Capelli polynomial, the characterization says that $\Gamma$ is $G$-graded verbally prime if and only if $\Gamma = Id_{G}(A)$ where $A$ is a finite dimensional $G$-graded simple algebra over $F$, an algebraically closed field. As for $T$-ideals $\Gamma$ that are $G$-graded strongly verbally prime, a characterization in terms of algebras is known only for $T$-ideals $\Gamma$ which contain a Capelli polynomial (or equivalently by \cite{AB}, $T$-ideals of identities of finite dimensional $G$-graded algebras). One such characterization is the following: A $G$-graded $T$-ideal $\Gamma$ containing a Capelli polynomial is strongly verbally prime if and only if $\Gamma = Id_{G}(A)$ where $A$ is a finite dimensional $G$-graded simple algebra which admits a $G$-graded division algebra form over a field $k$, and $k$ contains an algebraically closed field $F$.

It turns out, in contrast to trivially graded matrix algebras, that finite dimensional $G$-graded simple algebras may admit a $G$-graded division algebra form but at the same time may not admit $G$-graded division algebra forms over fields $k$ which contain $F$. In Theorem \ref{essentially verbally prime} we characterize $G$-graded $T$-ideals of identities $\Gamma = Id_{G}(A)$ where $A$ is finite dimensional $G$-graded simple which admits a $G$-graded division algebra form. These are precisely the $G$-graded $T$-ideals $\Gamma$ which contain a Capelli polynomial and are \textit{essentially verbally prime}.
\begin{defn}\label{essentially verbally prime definition}
We say a $G$-graded $T$ ideal $\Gamma$ is \textit{essentially verbally prime} if \linebreak $pq \in \Gamma \Rightarrow p \in \Gamma$ or $q \in \Gamma$ whenever $p$ and $q$ are $G$-homogeneous multilinear polynomials on disjoint sets of variables whose coefficients are in $k$, the unique minimal field of definition of $\Gamma$.
\end{defn}
For future reference we have
\begin{thm}\label{essentially verbally prime}
Let $A$ be a finite dimensional $G$-graded simple algebra over its $e$-center $F$. Then $A$ admits a $G$-graded division algebra form over its $e$-center if and only if $\Gamma = Id_{G}(A)$ is essentially verbally prime.
\end{thm}

Notions as $G$-graded Azumaya algebra, $G$-graded division algebra or $G$-graded Brauer group appear in the literature. This goes back to the fundamental paper of C.T.C. Wall on $\mathbb{Z}_{2}$-graded Brauer group followed with the work of C. Small and others. In these articles the authors require the group $G$ to be abelian, an assumption that seems to be essential in order to define a graded tensor product on Brauer classes (\cite{WALL}, \cite{Small}, \cite{Long}, \cite{Childs}, \cite{BereleRowen}, \cite{ElduqueKotchetov1}). The case where the group $G$ is not necessarily abelian is considered in structural and classification problems of $G$-graded simple algebras and in particular of $G$-graded division algebras (\cite{Karrer}, \cite{ElduqueKochetov}, \cite{BalbaMikhalev}, \cite{BahturinZaicev1}, \cite{BahturinZaicev2}, \cite{BEK}, \cite{ElduqueKotchetov1}). In this article we do not discuss the classification of $G$-graded division algebras nor (in case $G$ is an arbitrary finite group) the possibility to introduce an algebraic structure on equivalence classes of $G$-graded Azumaya algebras over a commutative ring $R$.

Here is the structure of the paper. The $2$nd section is devoted to the explicit determination of the unique minimal field of definition $k$ of the $G$-graded simple algebra $A$ and of its $T$-ideal of $G$-graded identities $Id_{G, F}(A)$ (Theorem \ref{field of definition theorem}). This allows us to construct the corresponding relatively free algebra $\mathcal{A}$ of $A$ over $k$. In the $3$rd section we show the algebra $\mathcal{A}$ can be localized by a suitable $e$-central polynomial $f$ so that the localized algebra $\mathcal{U} = f^{-1}\mathcal{A}$ is Azumaya in the sense of Artin-Procesi. In the $4$th section we show $\mathcal{U}$ satisfies the \textit{$1-1$ correspondence of ideals} condition. In the $5$th section we show $\mathcal{U}$ is generic and finally, in the $6$th section, we apply the construction of $\mathcal{U}$ in order to characterize $G$-graded simple algebras $A$ over $F$ which admit a $G$-graded division algebra form.

\section{Field of definition}

Let $A$ be a $G$-graded simple algebra over its $e$-center $F$ where $F$ is an algebraically closed field of characteristic zero. Let $\Gamma = Id_{G,F}(A)$ be the corresponding $T$-ideal of $G$-graded identities of $F\langle X_{G} \rangle$.

\begin{defn}
We say the $G$-graded $F$ algebra $A$ is defined over a field $k$ if there exists a $G$-graded $k$-algebra $A_{0}$, where $k$ is the $e$-center of $A_{0}$, such that $A_{0} \otimes_{k}F \cong A$ as $G$-graded algebras. Similarly, we say that $\Gamma$ is defined over a subfield $k$ of $F$ if there is a $T$-ideal $\Gamma_{0}$ of $k\langle X_{G} \rangle$ such that $\Gamma_{0}\otimes_{k}F = \Gamma$.
\end{defn}
Explicitly, an identity $p \in \Gamma = Id_{G,F}(A)$ is defined over a subfield $k$ if $p = \sum_{i}\gamma_{i}p_{i}$, where $\gamma_{i} \in F$ and the $p_{i}$'s are identities of $A$, with coefficients in $k$. Note that if $p$ admits such a decomposition we may assume $\gamma_{i}$ are linearly independent over $k$.
The following lemma is basic.

\begin{lem}\label{field of definition for algebra is also for identities}

If a $($$G$-graded$)$ $F$-algebra $A$ is defined over $k$, then the $T$-ideal of $($$G$-graded$)$ $F$-identities of $A$ is defined over $k$ as well.

\end{lem}
\begin{proof}

Let $A\cong A_{0}\otimes_{k}F$ for some $G$-graded $k$-algebra $A_{0}$.
If we write an identity of $A$ with coefficients in $F$ as a linear combination
$$
\gamma_1\times p_{1} +\gamma_2\times p_{2} + \cdots + \gamma_n\times p_{n}
$$
where the polynomials $p_{i}$ have coefficient in $k$ and the $\gamma_{i}$'s are linearly independent over $k$, then evaluating on $A_{0}$ we obtain a sum
$$
\gamma_1\otimes b_1 + \gamma_2\otimes b_2 + \cdots + \gamma_n\otimes b_{n}
$$
in $F\otimes_{k}A_{0}$. Because the $\gamma_{i}$'s are linearly independent over $k$ all $b_{i}$'s must be zero.

\end{proof}

This simple observation allows us to show that certain algebras $A$ do not admit a form over some field $k_{0}$, that is there is no $B = B_{k_{0}}$ over $k_{0}$ such that $A = B_{F} = B_{k_{0}}\otimes_{k_{0}}F$. The idea is to exhibit a multilinear polynomial $f$ with coefficients in $F$ which on the one hand it is an identity of $A$, and on the other hand if it is written as a sum $f = \gamma_1\times p_{1} +\gamma_2\times p_{2} + \cdots + \gamma_n\times p_{n}$
in $F\otimes_{k_{0}}k_{0}\langle X_{G} \rangle$ where the $\gamma_{i}$'s are linearly independent over $k_{0}$, then at least one $p_{i}$ is a nonidentity on $A$. Indeed, this will show that such a $B$ does not exist for if it did, $f$ would clearly be an identity of $B$ and hence every $p_{i}$ would be an identity of $B$. Then by multilinearity, every $p_{i}$ would be an identity of $A$.

Before presenting the construction of the unique minimal field of definition of $A$ and $Id_{G,F}(A)$ we recall some basic facts about finite dimensional $G$-graded simple algebras over an algebraically closed field $F$ of characteristic zero. We shall assume $G$ is finite although some of the constructions hold also for infinite groups.

If $A$ is finite dimensional $G$-graded simple it admits a presentation \linebreak $P_{A}= \{H, \alpha, \mathfrak{g} = (g_{1},\ldots,g_{n})\}$ where $H$ is a subgroup of $G$, $\alpha$ is a $2$-cocycle of $H$ with coefficients in $F^{*}$ and the $n$-tuple $\mathfrak{g} = (g_{1},\ldots,g_{n})$ consists of elements, possibly with repetitions, taken from a fixed right transversal of $H$ in $G$ which we denote by $T$. With this presentation
$$A \cong F^{\alpha}H \otimes_{F} M_{n}{(F)}$$ and $A_{g} = Span_{F}\{u_{h} \otimes e_{i,j}: g = g_{i}^{-1}hg_{j}\}$, where $g \in G$. Here, $A_{g}$ denotes the $g$-homogeneous component of $A$. In particular, the presentation $P_{A}$ determines the isomorphism type of the $G$-graded algebra $A$ (see \cite{BSZ} for details).

We recall also that a $G$-graded simple algebra $A$ may admit different presentations $P_{A}$. In fact there exists a set of admissible (or basic) moves on presentations, such that on the one hand two presentations that are obtained one from the other by means of an admissible move represent $G$-graded isomorphic algebras and on the other hand every two presentations that represent $G$-graded isomorphic algebras are obtained one from the other by means of a finite sequence of admissible moves. The admissible moves are classified into three \textit{type}, referred as moves of type $I$, $II$ and $III$ (see \cite{AljHaile}). For the reader convenience we recall them here.

\textit{Type} $I$: Multiplying an element in the tuple $\mathfrak{g}$ by elements of $H$ on the left.

\textit{Type} $II$: Permuting the elements of the tuple $\mathfrak{g}$.

\textit{Type} $III$: Multiplying the tuple $\mathfrak{g}$ by some $g \in G$, replacing $H$ by $H^{g}= gHg^{-1}$ and replacing the $2$-cocycle $\alpha$ by $\alpha^{g}$ where $\alpha^{g}(h_{1}, h_{2}) = \alpha(g^{-1}h_{1}g, g^{-1}h_{2}g)$.

By means of the basic moves we may normalize the tuple $\mathfrak{g}$ in the presentation $P_{A}= \{H, \alpha, \mathfrak{g} = (g_{1},\ldots,g_{n})\}$ so that it satisfies the following conditions:

\begin{enumerate}
\item
The element $e$ appears in the tuple $\mathfrak{g}$ and with maximal number of repetitions.

\item
Equal representatives are adjacent to each other in $\mathfrak{g}$.

\item

Representatives that normalize $H$ appear earlier in $\mathfrak{g}$.
\end{enumerate}

Before we proceed with the construction of the minimal field of definition let us make the following convention.

\begin{convention}
Given a $G$-simple algebra $A$ we shall always fix a presentation $P_{A}$ and a basis of homogeneous elements of the form $u_{h} \otimes e_{i,j}$. All polynomials in $F\langle X_{G}\rangle$ we consider will be multilinear and all their evaluations on $A$ will be on elements of the fixed basis.
\end{convention}

Let $N_{G}(H)$ denote the normalizer of $H$ in $G$. Note that any element of $N_{G}(H)$ acts on $M(H)$, the Schur multiplier of $H$.
For the given $[\alpha] \in H^{2}(H, F^{*})$, we denote by $\pi(\alpha)$ (rather than $\pi([\alpha])$) the image of $[\alpha]$ by the (surjective) map induced by the Universal Coefficient Theorem
$$
\pi: H^{2}(H, F^{*})\rightarrow Hom(M(H), F^{*}).
$$
and let $B_{\alpha} = ker(\pi(\alpha))\subseteq M(H)$.

We denote by $N$ the elements of $N_{G}(H)$ which normalize $B_{\alpha}$. We see that if the group $\mu_{n}$ of all $n$th roots of unity is precisely the image of $\pi(\alpha)$ in $F^{*}$, from the isomorphism of $M(H)/B_{\alpha} \cong \mu_{n}$ we obtain an induced action of $N$ on $\mu_{n}$ and hence on $\mathbb{Q}(\mu_{n})$. As we shall see below, the minimal field of definition of $A$ is a field $k$ with $\mathbb{Q}(\mu_{n})^{N} \leq k \leq \mathbb{Q}(\mu_{n})$.

We now come to an important lemma which characterizes subgroups of $N_{G}(H)$ that can be \textit{factored out} from the $n$-tuple $\mathfrak{g}$ of $H$-right cosets representatives. A close inspection on the moves of type $I,II,III$, shows that moves of type $I$ or $II$ on $P_{A}$ do not change the subgroup $H$. Furthermore, the subgroup $H$ is not changed also if we apply a move of type $III$ as long as $g \in N_{G}(H)$.

Let $\Lambda = \Lambda_{P_{A}}$ be the multiset (with repetitions) consisting of the $H$-right cosets represented by the elements of $\mathfrak{g} \in P_{A}$. We see that if we apply moves of type $I$ and $II$ on $P_{A}$, the multiset $\Lambda$ remains invariant. Applying a move of type $III$ with an element $g \in N_{G}(H)$ sends right $H$-cosets to right $H$-cosets, however it may change the multiset $\Lambda$. We therefore consider the family $\Omega_{G,H,n}$ of multisets of cardinality $n$ of $H$-right cosets in $G$. Clearly $\Lambda \in \Omega_{G,H,n}$. Multiplication on the left
$$
Hb \mapsto gHb = Hgb
$$
of the right $H$-cosets appearing in $P \in \Omega_{G,H,n}$ determines an action of $N_{G}(H)$ on $\Omega_{G,H,n}$ and we let $K$ be the subgroup of $G$ which fixes $\Lambda$. We say the elements $g \in K$ normalize $\Lambda$ and write $g \times \Lambda = \Lambda$. We have $H \unlhd K \leq N_{G}(H) \leq G$.

In order to have a better insight on the group $K \leq N_{G}(H)$ let us present two additional characterizations of it. Let $U$ be a subgroup of $N_{G}(H)$ which contains $H$.

\begin{defn}
We say $U$ has the \textit{equal repetition property} (or the \textit{equal frequency property}) in $\Lambda$ if any two $H$-cosets in the tuple $\Lambda$ that determine the same $U$-coset, they appear the same number of times in $\Lambda$.  So in fact for such $U$ we have that all $H$-cosets contained in $U$ are represented the same number of times in $\Lambda$.
\end{defn}

Next, we introduce the \textit{factorization condition}. As above we let $H \leq U \leq N_{G}(H)$.

\begin{defn} We say $U$ can be \textit{factored out} from $\Lambda$ (or  equivalently the multiset $\Lambda$ can be factored out along $U$) if the following hold: If $\mathfrak{U} = \{e, u_{i_2}, \ldots, u_{i_p}\}$  is a transversal of $[U:H]$, then $\Lambda$ can be expressed as
$$
\{e, u_{i_2}, \ldots, u_{i_p}\}\times \Lambda_{0}
$$
where $\Lambda_{0}$ is a multiset of $H$-right cosets.

\end{defn}
Note that if the factorization condition holds for some transversal of $[U:H]$ then it holds for every transversal.
\begin{example}
Let $G \cong D_{6} = \{e, \sigma, \sigma^{2}, \tau, \sigma\tau, \sigma^{2}\tau \}$, $H = \{e\}$ and
$$\mathfrak{g} = (e, e, \tau, \tau, \sigma, \sigma, \sigma^{2}\tau,  \sigma^{2}\tau) =  (e, \tau)\times (e, e, \sigma, \sigma^{2}\tau).$$
It is clear that left multiplication by $\tau$ preserves the tuple whereas multiplication on the left by $\sigma^{2}\tau$ does not.
\end{example}
We have

\begin{lem}\label{three equivalent conditions}
Notation as above. For a group $U = \{e,u_{2},\ldots,u_{q}\}$ with $H \leq U \leq N_{G}(H)$ the following conditions are equivalent.
\begin{enumerate}
\item
$U$ normalizes $\Lambda$ via left multiplication.

\item

$U$ satisfies the equal frequency property.

\item

The multiset $\Lambda$ can be factored along $U$.
\end{enumerate}
\end{lem}

\begin{proof}
It is clear that if $U$ can be factored out from $\Lambda$ then any element of $U$ normalizes $\Lambda$, hence $(3 \Rightarrow 1)$. For the proof of $(1 \Rightarrow 2)$ suppose $U$ normalizes $\Lambda$ and we assume as we may that $Hg_{1}$ and $Hg_2$ are two $H$-cosets contained in $U$. Suppose by contradiction they appear in $\Lambda$ with frequencies $d_1$ and $d_2$ where $d_{1} \neq d_{2}$. Clearly there is $u \in U$ such that $ug_{1} = g_{2}$ and hence $uHg_{1} = Hug_{1} = Hg_{2}$ which shows that the multiplication of $\Lambda$ by $u$ on the left sends the right $H$ cosets represented by $g_1$ to the cosets represented by $g_2$. Because they have different frequency, $u$ does not normalize $\Lambda$. Contradiction.

$(2 \Rightarrow 3)$: Suppose $U$ satisfies the equal frequency property. We want to show the multiset $\Lambda$ can be decomposed into disjoint sets $\Lambda_{i}$ where each $\Lambda_{i}$ has the form $\{e, u_{2},\ldots, u_{q}\}Hb = \{H, Hu_{2},\ldots, Hu_{q}\}b$. This will follow if we show $Hb$ and $Hub$ appear with the same frequency in $\Lambda$ for every $u \in U$. But this is exactly our assumption because $Hb$ and $Hub$ determine the same right coset of $U$ in $G$.
\end{proof}
In view of the first condition of the lemma we have
\begin{cor}\label{Definition of the subgroup K}
Notation as above.  There exists a unique maximal subgroup $K$ with $H \lhd K \leq N_{G}(H)$, that satisfies the conditions of Lemma \ref{three equivalent conditions}.
\end{cor}

We now return to the subgroup $B_{\alpha}$ of $M(H)$ we considered above and let $S$ be the subgroup of $K$ that normalizes $B_{\alpha}$, that is $S = N\cap K$. We have the following sequence of subgroups
$$
H \unlhd S \leq K \leq N_{G}(H) \leq G.
$$
Recalling that $N$ acts on $\mathbb{Q}(\mu)$, we have an action of $S$ and we let $k = \mathbb{Q}(\mu)^{S}$.
The main result of this section is the following.
\begin{thm}\label{minimal field of definition Algebra and ideal}
$(1)$ The field $k$ is the unique minimal field of definition of the $G$-graded simple algebra $A$. $(2)$ The field $k$ is the unique minimal field of definition of $Id_{G,F}(A)$.
\end{thm}

We know by Lemma \ref{field of definition for algebra is also for identities} that if $L$ is a field of definition of $A$ it is also a field of definition of $Id_{G,F}(A)$. So, in order to prove the theorem we will show that (a) $k$ is a field of definition of $A$ (b) any field $L$ which does not contain $k$ is not a field of definition of $Id_{G,F}(A)$. For the proof of (a) we claim it is sufficient to show the $S$-simple algebra $A_{S}$ with presentation $P_{A_{S}}= \{H, \alpha, \mathfrak\{g_{S} = (s_{1},\ldots,s_{d})\}$, where $S = \{s_1, \ldots, s_{d}\}$, is defined over $k$ (Warning: $A_{S}$ is not the full $S$-homogeneous component of $A$). To see this, let
$\Lambda = (s_{1},\ldots,s_{d})(Hb_{1},\ldots,Hb_{r})$ be a decomposition of the tuple $\Lambda$ along $S$ and let
let $B_{S, k}$ be a $k$-form of $A_{S}$ as an $S$-graded algebra. Consider the algebra $M_{r}(k)\otimes_{k} B_{S, k}$ where the homogeneous degree of $e_{i,j} \otimes z_{\sigma}$, $z_{\sigma} \in (B_{S, k})_{\sigma}$, is $b_{i}^{-1}\sigma b_{j} \in G$. Extending scalars to $F$ yields
\begin{equation}\label{A_S}
F \otimes_{k}M_{r}(k)\otimes_{k} B_{S, k} \cong M_{r}(k)\otimes_{k}F \otimes B_{S, k} \cong M_{r}(F) \otimes_{F} A_{S} \cong A.
\end{equation}

Let us show now the algebra $A_{S}$ is defined over $k$. Let $T$ be the kernel of the action of $S$ on $\mathbb{Q}(\mu_{n})$ and $S/T = Gal(\mathbb{Q}(\mu_{n})/k)$. We exhibit elements which generate the $k$-form $B_{S, k}$ of $A_{S}$ over $k$. Because the cohomology class $\alpha$ gets values in $\mathbb{Q}(\mu_{n})$ we know the algebra $A_{S}$ has an obvious $\mathbb{Q}(\mu_{n})$-form, which we still denote by $A_{S}$. Decompose $A_{S} \cong \mathbb{Q}(\mu_{n})^{\alpha}H \otimes_{\mathbb{Q}(\mu_{n})}M_{tm}(\mathbb{Q}(\mu_{n}))$ where $t$ is the order of $T/H$ and $m$ the order of $S/T$. We may decompose the matrix algebra $M_{tm}(\mathbb{Q}(\mu_{n}))$ and get
$$
A_{S} \cong \mathbb{Q}(\mu_{n})^{\alpha}H \otimes_{\mathbb{Q}(\mu_{n})}M_{t}(\mathbb{Q}(\mu_{n}))\otimes_{\mathbb{Q}(\mu_{n})} M_{m}(\mathbb{Q}(\mu_{n}))
$$
where $M_{t}(\mathbb{Q}(\mu_{n}))$ is elementary graded by a transversal of $H$ in $T$ and $M_{m}(\mathbb{Q}(\mu_{n}))$ is elementary graded by a transversal of $T$ in $S$.
The form we are looking for is the algebra over $k$ generated by the following sets:

\begin{enumerate}
\item
$$
\{\sum_{i}u_{h}^{g_{i}^{-1}}\otimes I_{t} \otimes e_{i,i}: h \in H \}
$$
\item
$$
\{\sum_{i}z^{g_{i}^{-1}}\otimes I_{t} \otimes e_{i,i}: z \in \mathbb{Q}(\mu_{n}) \} (^{*})
$$
\item
$$
\{1 \otimes e_{i,j} \otimes I_{m}: i, j=1,\ldots,r \}
$$
\item
$$
\{1 \otimes I_{t} \otimes P_{\sigma}: \sigma \in Gal(\mathbb{Q}(\mu_{n})/k) \}
$$
where $P_{\sigma}$ is the $m\times m$ permutation matrix, homogeneous of degree $\sigma$.
\end{enumerate}
$(^{*})$ Note that the $2$nd set above is obtained from the $1$st set so in fact it can be omitted from the generating set.

We proceed now with the second part of the proof, namely for any field $L$ not containing $k$ we can find an identity of $A$ that is not defined over $L$.
It is convenient to regard $\mathfrak{g}$ as an ordered tuple. Our first task is to get a better understanding of the ordered tuples one may obtain by left multiplication of $\mathfrak{g}$ by an element of $N_{G}(H)$.

\begin{defn}
Consider the $n$-tuple $\mathfrak{g} = \{g_{1} = e ,g_{2},\ldots,g_{n}\} \in G^{n}$ and the $n$-tuple $\Lambda = \{Hg_{1},Hg_{2},\ldots,Hg_{n}\}$ of right $H$-cosets represented by $\mathfrak{g}$. For an $n$-tuple $(h_{1},\ldots,h_{n}) \in H^{n}$ we refer to
$$
\mathfrak{g}'_{(h_{1},\ldots,h_{n})} = (g_{1}^{-1}h_{1}g_{2}, g_{2}^{-1}h_{2}g_{3},\ldots, g_{n-1}^{-1}h_{n-1}g_{n}, g_{n}^{-1}h_{n}g_{1})
$$
as the \textit{derivative} of $\mathfrak{g}$ \textit{along} $(h_{1},\ldots,h_{n})$. Likewise, we refer to the $n$-tuple of sets
$$
\Lambda' = (g^{-1}_{1}Hg_{2}, g_{2}^{-1}Hg_{3},\ldots, g_{n-1}^{-1}Hg_{n}, g_{n}^{-1}Hg_{1})
$$
as the \textit{full derivative} of $\Lambda$ along $H^{n}$. Clearly, $\Lambda'$ is the set of all derivatives of $\mathfrak{g}$ along the elements of $H^{n}$.

\end{defn}

It is easy to see that if we multiply the $n$-tuple $\Lambda$ by an element $\theta \in N_{G}(H)$ on the left, the obtained tuple ${\hat{\Lambda}} = \theta\times {\Lambda}$ has the same full derivative as ${\Lambda}$. We claim the converse holds.

\begin{lem}\label{equal derivative I}
Let $\Lambda$ and $\hat{\Lambda}$ be $n$-tuples of right $H$-coset. If $\Lambda' = \hat{\Lambda}'$ then there exists $\theta\in N_{G}(H)$ such that $\hat{\Lambda} = \theta\times \Lambda$.
\end{lem}
\begin{proof}
From the condition on the derivatives we have that $g_{1}^{-1}Hg_{j} = \hat{g}_{1}^{-1}H\hat{g}_{j}$ and so $\theta = \hat{g}_{1}g_{1}^{-1} \in N_{G}(H)$ because it determines a left coset which is also a right coset and hence it normalizes $H$. Moreover, we obtain that $\hat{g}_{1}g_{1}^{-1}g_{j}$ and $\hat{g}_{j}$ represent the same $H$-right coset as desired.
\end{proof}

We need a stronger lemma regarding the derivative of the $n$th tuple \linebreak $\mathfrak{g}=(g_{1},\ldots,g_{n})$. The proof is contained in the statement.

\begin{lem}\label{equal derivative II}
Let $\mathfrak{g}_{id} = \mathfrak{g}$ be an $n$-tuple of right $H$-coset representatives and let $\mathfrak{g}_{\sigma}$ be the permutation of $\mathfrak{g}$ by $\sigma \in Sym(n)$. Suppose $\mathfrak{g}_{id}$ and $\mathfrak{g}_{\sigma}$ have a \textit{common derivative}, that is there are tuples $\vec{h}=(h_{1},\ldots,h_{n}), \vec{h}'= (h'_{1},\ldots, h'_{n})$ such that $\mathfrak{g}'_{id, \vec{h}}:=(g_{1}^{-1}h_{1}g_{2},\ldots,g_{n}^{-1}h_{n}g_{1} )= \mathfrak{g}'_{{\sigma, \vec{h'}}} := (g_{\sigma(1)}^{-1}h'_{1}g_{\sigma(2)},\ldots,g_{\sigma(n)}^{-1}h'_{n}g_{\sigma(1)})$. Then $g_{\sigma(1)}g_{1}^{-1}h_{1}\cdots h_{i} g_{i} = h'_{1}\cdots h'_{i} g_{\sigma(i)}$. In particular, if $g_{\sigma(1)}g_{1}^{-1}\in N_{G}(H)$, the multiplication of the tuple $\mathfrak{g}$ by $u = g_{\sigma(1)}g_{1}^{-1}$ determines the same tuple of $H$-right cosets as $\mathfrak{g}_{\sigma}$.
\end{lem}

We now come to the construction of a polynomial identity $p_{L}$ of $A$, where $L$ is any field not containing $k$. As mentioned above the polynomial $p_{L}$ we construct is defined (as an identity of $A$) over the field $k$ but it is not defined over $L$.

We start by putting a special ordering on the $n$-tuple $\mathfrak{g} = (g_{1},\ldots,g_{n})$ (the first two items below were already mentioned earlier).
\begin{enumerate}

\item
By means of the basic moves we arrange equal representative to be adjacent in $\mathfrak{g}$.
\item

We let $g_{1} = e$ and arrange so that it appears in $\mathfrak{g}$ with maximal multiplicity, say $d \geq 1$.
\item
Temporarily we order the coset representatives with multiplicities in decreasing order. We change the notation as follows:
$$\mathfrak{g} = (g_{[1, d_{1}]} = e_{d_{1}}, g_{[2, d_{2}]},\ldots, g_{[r, d_{r}]})$$
where (now) $g_{1},\ldots,g_{r}$ are the different $H$-coset representatives, $g_{[i, d_{i}]} = (g_{i},\ldots,g_{i})$, $d_{i}$ times and $(d_{1}\geq d_{2} \geq \ldots \geq  d_{r})$. We have $\sum^{r}_{i=1}d_{i} = n.$
\item
Factor out from the tuple $\mathfrak{g}$ the $H$-coset representatives from the group $K$ (see Lemma \ref{three equivalent conditions} and Corollary \ref{Definition of the subgroup K}), starting, i.e. put them on the left, with the representatives from $S$. Thus, we get the following configuration (again, changing notation).
$$
\mathfrak{g} = (e,s_{2},\ldots,s_{p}, k_{p+1},\ldots,k_{m})\times (e_{d_{1}}, g_{[2, d_{2}]},\ldots, g_{[q, d_{q}]})
$$
where $g_{[j, d_{j}]}$ stands for $(g_{j},\ldots,g_{j})$, $d_{j}$ times.
We assume as we may, that the elements $(e, g_{2},\ldots,g_{q})$ are different right $K$-coset representatives in $G$.
With this notation we have $p = [S:H]$, $m = [K : H]$, all elements of $K$ appear with frequency $d = d_{1}$. The order of $\mathfrak{g}$ is $n = m\times (d_{1}+ \cdots + d_{q})$.
\item
We remind the reader that $K$ is the maximal subgroup of $N_{G}(H)$ that can be factored out from $\mathfrak{g}$. Therefore, it may be the case that among the elements $g_{2}, \ldots, g_{q}$ there are representatives which belong to $N_{G}(H)$. If we have $t \geq 1$ such elements (including $e$) we write
$$
\mathfrak{g} = (e,s_{2},\ldots,s_{p}, k_{p+1},\ldots,k_{m})\times (e_{d_{1}}, \nu_{[2, d_{2}]},\ldots, \nu_{[t, d_{t}]}, g_{[{t+1},d_{t+1}]}\ldots, g_{[{q},d_{q}]}).
$$

In the next steps we treat the representatives $(g_{[{t+1},d_{t+1}]}\ldots, g_{[{q},d_{q}]})$, namely those that are outside $N_{G}(H)$. To this end note that if $\nu\in N_{G}(H)$ and $g \in G$ is an element with $H\nu \neq Hg$ then $\nu^{-1}Hg \cap H = \varnothing$. On the other hand, for $g_{i} \not \in N_{G}(H)$ this is false in the following strong sense: There exists $g \in G$, possibly not in $\mathfrak{g}$, with $Hg_{i} \neq Hg$ and yet $g_{i}^{-1}Hg \cap H \neq \varnothing$. Indeed, because $g_{i}$ does not normalize $H$, $g^{-1}_{i}Hg_{i} \neq H$. It follows there is $g \in G$ with $g_{i}^{-1}Hg \neq g_{i}^{-1}Hg_{i}$ which intersect nontrivially $H$. We introduce an equivalence relation on the right $H$-coset representatives in $\mathfrak{g}$.
\begin{defn}
We say $g, g' \in \mathfrak{g}$ are related if $g^{-1}Hg' \cap H \neq \varnothing$.
\end{defn}
Let $[g]$ be the equivalence class represented by $g \in \mathfrak{g}$. By the previous paragraph we see that $H$-coset representatives which belong to $N_{G}(H)$ represent classes of cardinality one. We refer to such classes as \textit{singletons} and \textit{nonsingletons} otherwise. Because some $H$-cosets may not be represented in $\mathfrak{g}$, we may find among the classes $[g_{i}]$ where $g_{i} \notin N_{G}(H)$, singletons as well as nonsingletons. Thus, with the above notation the nonsingletons classes are contained in $(g_{[{t+1},d_{t+1}]}\ldots, g_{[{q},d_{q}]}).$ We can set now the final refinement on the ordering of the tuple $\mathfrak{g}$, namely
$$
\mathfrak{g} = (e,s_{2},\ldots,s_{p}, k_{p+1},\ldots,k_{m})\times
$$
$$(e_{d_{1}}, \nu_{[2, d_{2}]},\ldots, \nu_{[t, d_{t}]}, g_{[{t+1},d_{t+1}]}\ldots, g_{[{f},d_{f}]}, g_{[{f+1},d_{f+1}]}, \ldots, g_{[{q},d_{q}]})
$$
where $g_{i}$, $t+1 \leq i \leq f$, represent singletons and nonsingletons if $f+1 \leq i \leq q$.
\end{enumerate}
Next, we want to describe the homogeneous degrees of the blocks determined by $\mathfrak{g}$. By definition of the $G$-grading, as determined by $P_{A}$, we see that the $n\times n$ matrix algebra is decomposed into blocks which arise from repetitions of the $H$-coset representatives. Note that if a representative appears $d_{i}$ times it determines a $d_{i} \times d_{i}$-diagonal block of $e$-homogeneous elements which may be denoted by $u_{e}\otimes e_{[k_{a}\times i_{l}, k_{a}\times i_{s}]}$ where $g_{i_l} = g_{i_s}$ and $k_{a}$ is any element of $K$. Note that the $e$-elements may appear only in this way, namely if $u_{h} \otimes e_{[k_{a}\times i_{l}, k_{b} \times j_t]}$ is homogeneous of degree $e$, then $h = e$, $a=b$ and  $g_{i_l} = g_{j_t}$ which implies $i = j$. We decompose the set of all basis elements $\{u_{h}\otimes e_{[k_{a}\times i_{l}, k_{b} \times j_{t}]}\}$ into blocks, where the $[k_{a}\times i, k_{b} \times j]$-block is of order $d_{i} \times d_{j}$, that is $d_{i}$-rows and $d_{j}$ columns, and with \textit{depth} $(= ord(H))$. The block is spanned by the basis elements $u_{h}\otimes e_{[k_{a} \times i_{l}, k_{b} \times j_{t}]}$, $l=1,\ldots, d_{i}$, $t = 1,\ldots, d_{j}$ and $h \in H$. It is convenient to regard a block as above as a $3$-dimensional \textit{box}, of order $(d_{i}, d_{j}, ord(H))$, as the $[k_{a} \times i, k_{b} \times j]$-block with $ord(H)$ pages.

As mentioned earlier, the $e$-component of $A$ appears only in diagonal blocks. Moreover, we see that the homogeneous degrees which appear in the box which corresponds to a diagonal block are all in $H$ if and only if the representative $g \in \mathfrak{g}$ normalizes $H$. In that case we say the $[k_{a}\times i, k_{a}\times i]$-diagonal block has pages only in $H$.

Now, take two arbitrary basis elements. Their multiplication
$$u_{h_{1}}\otimes e_{[k_{a_1} \times i_{l}, k_{b_1} \times j_{t}]}\times u_{h_{2}}\otimes e_{[k_{a_2} \times x_{f}, k_{b_2} \times y_{s}]}$$
is nonzero if and only if $[k_{b_1} \times j_{t}] =  [k_{a_2} \times x_{f}]$. On the other hand, it is clear that any two basis elements $z$ and $w$ can be bridged by a basis element, that is there exists $c$ with $zcw \neq 0$. The next lemma explains the reason for introducing the terminology of singletons and nonsingletons.

\begin{lem}\label{bridges of singletons and nonsingletons}
Consider diagonal blocks represented by representatives $k_{a}g_{i_{l}}$ and $k_{b} g_{j_{t}}$ in $\mathfrak{g}$ and let $z_e$ and $w_e$ be basis elements of $A$, homogeneous of degree $e$, which belong to these blocks respectively. Then the following hold:

\begin{enumerate}
\item
Suppose $z_{e}$ and $w_{e}$ belong to the same diagonal block. Then, there exists an $e$-element $c_{e}$ such that $z_{e}c_{e}w_{e} \neq 0$. Furthermore, if $c_{g}$ is homogeneous of degree $g$ such that $z_{e}c_{g}w_{e} \neq 0$, then necessarily $g = e$.
\item
Suppose $k_{a}g_{i_{l}}$ and $k_{b} g_{j_{t}}$ are not related by the equivalence relation (e.g. one of them is a singleton) and suppose $z_{e}c_{g}w_{e} \neq 0$, some $g \in G$. Then necessarily $g \notin H$.
\item
Suppose the diagonal blocks are different and are represented by elements $k_{a}g_{i_{l}}$ and $k_{b}g_{j_{t}}$ which are related by the equivalence relation. Then there is $c_{h}$, $h \in H$ with $z_{e}c_{h}w_{e} \neq 0$. However, there is no such bridge $c_{h}$ with $h = e$.
\end{enumerate}
\end{lem}
\begin{proof}
The proofs of parts $1$ and $2$ follow directly from the definitions. Let us prove part $3$: Clearly, any bridge between different blocks cannot be homogeneous of degree $e$. By assumption we have that $(k_{a}g_{i_{l}})^{-1}Hk_{b}g_{j_{t}} \cap H \neq \varnothing$, so if $h' = (k_{a}g_{i_{l}})^{-1}\times h\times k_{b}g_{j_{t}}$. It follows that the basis element $u_{h} \otimes e_{[k_{a}\times g_{i_{l}}, k_{b}\times g_{j_{t}}]}$ is homogeneous of degree $h'$ and bridges $z_{e}$ and $w_{e}$.

\end{proof}

Let $L$ be a field not containing $k$. As mentioned earlier, our goal is to construct a polynomial identity $p_{L}$ of $A$ which is defined over $k$ but not over $L$.

Let
$$\Delta = \{u_{h}\otimes e_{[k_{a} \times i_{l}, k_{b} \times j_{t}]}\}$$ be a basis of $A$ and let $$\Delta_{e} = \{u_{e}\otimes e_{[k_{a} \times i_{l}, k_{a} \times i_{t}]}\}$$ be the subset of $\Delta$ consisting of all $e$-elements. Recall that the elements of $\Delta_{e}$ appear in diagonal blocks represented by $k_{a}g_{i_l}$.

We consider a \textit{nonzero} product $E_{0}$ of basis elements of the form $u_{h}\otimes e_{[k_{a} \times i_{l}, k_{b} \times j_{t}]}$ which visits \textit{all} basis elements of $\Delta_{e}$.

For each element $\{u_{e}\otimes e_{[k_{a} \times i_{l}, k_{a} \times i_{t}]}\} \in \Delta_{e}$ we choose an appearance of it in the product $E_{0}$ (we know each one of these appears at least once) and refer to it as a designated element of $E_{0}$.
\begin{rem}\label{Euler path}
In fact it is possible to arrange that visits in the $e$-diagonal blocks consist only of designated elements but this will not play a role in the construction.
\end{rem}

We assume as we may, the product satisfies the following conditions
\begin{enumerate}
\item
Elements of $\Delta_{e}$ which belong to the same $e$-block are separated exclusively by $e$-elements.
\item
(For convenience) The $e$-blocks appear in $E_{0}$ successively along the main diagonal, that is the product starts with the $e$-block in place $(1,1)$, $(2,2)$ and so on. In particular we visit each diagonal block exactly once.
\item
Multiplying the entire monomial $E_{0}$ on the left and on the right by suitable basis elements we may assume the monomial does not start and does not end with a designated element. Moreover we may assume the monomial $E_{0}$ starts and ends with an element in $\Lambda_{e}$.
\item
The bridges $w_{g}$ between $e$-blocks are chosen as in Lemma \ref{bridges of singletons and nonsingletons}.
\end{enumerate}

Consider the homogeneous degrees of all basis elements which appear in $E_{0}$ and let $Z_{0}$ be a multilinear monomial whose indeterminates are homogeneous of degrees corresponding to the elements of $E_{0}$. Clearly, by its very construction, the monomial $Z_{0}$ has a nonzero evaluation $\hat{Z}_{0} = E_{0}$. Note however that the monomial $Z_{0}$ may admit many different nonzero evaluations.

Now, we want to extend the monomial $E_{0}$ by bordering the designated elements with $e$-diagonal elements. These extra elements are called \textit{frames}. We denote the monomial obtained by $E_{1}$.

We denote by $Z_{1}$ the multilinear monomial corresponding to $E_{1}$ and denote by $\hat{Z}_{1} = E_{1}$. The designated variables of $Z_{1}$ are those $e$-variables, denoted by $X$, that correspond to the designated $e$-elements in $E_{1}$. We denote by $X^{\diamond}$ the extra $e$-variables that correspond to $e$-elements in $E_{0}$ (see Remark \ref{Euler path}), by $Y$ the $e$-variables that correspond to the frame elements in $E_{1}$ and by $W$ the necessarily non-$e$-variables that correspond to the $e$-block bridges.

We let $p_{1} = p_{1}(X, X^{\diamond}, Y, W)$ be the polynomial obtained by alternating the designated variables
$$
p_{1} = \sum_{\sigma \in Sym(d_{e})}(-1)^{\sigma}Z_{1, \sigma}
$$
where $d_{e}$ is the cardinality of $\Lambda_{e}$.
Here are some observations regarding the polynomial $p_{1}$ and its evaluations on $\Lambda$.
\begin{enumerate}
\item
Because of the alternation, the designated variables $X$ must be evaluated on the full set $\Lambda_{e}$, the basis elements of $A_{e}$, as long as the value is nonzero.
\item
For each evaluation at most one monomial, say $Z_{1, \sigma}$, obtains a nonzero evaluation. Indeed, the values of the designated variables $X$ determine uniquely the values of the frames variables $Y$ if the value of the monomial is nonzero. It follows that if the value of the frames are compatible with the value of the designated variables of the monomial $Z_{1, \sigma}$, they are not compatible with the value of the designated variables of the monomial $Z_{1, \tau}$ for $\tau \neq \sigma$ and hence the value of $Z_{1, \tau}$ is zero.
\item
From the construction it follows that $Z_{1} = Z_{1,id}$ has a nonzero evaluation and hence $p_{1}$ is a nonidentity (because, by the previous statement, the other monomials vanish).
\end{enumerate}

Our next task is to insert variables/polynomials so that, roughly speaking, they \textit{reduce} the possible monomials $Z_{1, \sigma}$, $\sigma \in Sym(d_{e})$ and $d_{e} = card(\Lambda_{e})$, which may get a nonzero value.
Let $p_{1} = p_{1}(X, X^{\diamond}, Y, W)$.  As mentioned above, by the alternation, in any nonzero evaluation we must evaluate the set of designated variables $X$ by a full basis of $A_{e}$. Taking the \textit{basic evaluation} of $X$ as in $E_{1}$ we obtain $p_{1}(\hat{X}, X^{\diamond}, Y, W)$, and clearly we can choose evaluations of the sets $Y$, $X^{\diamond}$ and $W$ (again, as in $E_{1}$), so the first monomial does not vanish and hence so does the entire polynomial. It is not difficult to construct examples such that by changing the values of $Y$, $X^{\diamond}$ and $W$ (and keeping the evaluation $\hat{X}$ of $X$) a different monomial, say $Z_{1, \sigma}$, $\sigma\neq id$, may get a nonzero value. We note on the other hand, that for the polynomial $p_{1}(\hat{X}, X^{\diamond}, Y, W)$, there may be monomials $Z_{1, \tau}(\hat{X}, X^{\diamond}, Y, W)$, which vanish for every evaluation of $Y$, $X^{\diamond}$ and $W$.

Let us illustrate this with a simple example.

Consider the algebra of $3 \times 3$-matrices, $G = \{e,g\}$-graded with the elementary grading determined by the tuple $\mathfrak{g} = (e,e,g)$. We consider the monomial $E_{0}$ (in the above notation) $E_{0} = e_{1,1}e_{1,2}e_{2,2}e_{2,1} \times e_{2,3} \times e_{3,3}$. We multiply the monomial on the left and on the right by $e$-elements, so we get $E_{0} = e_{1,1}\textbf{e}_{1,1}\textbf{e}_{1,2}\textbf{e}_{2,2}\textbf{e}_{2,1} \times e_{2,3} \times \textbf{e}_{3,3}e_{3,3}$. We insert the frames bordering the designated elements and get the monomial
$$
E_{1} = e_{1,1}e_{1,1}\cdot\textbf{e}_{1,1}\cdot e_{1,1}\cdot \textbf{e}_{1,2}\cdot e_{2,2}\cdot\textbf{e}_{2,2}\cdot e_{2,2}\cdot\textbf{e}_{2,1}\cdot e_{1,1} \times e_{2,3} \times e_{3,3}\cdot\textbf{e}_{3,3}\cdot e_{3,3}e_{3,3}.
$$
Next we construct the multilinear monomial with variables corresponding to the elements in $E_{0}$

$$Z_{1} = w_{e,1}y_{e,1}x_{e,1}y_{e,2}x_{e,2}y_{e,3}x_{e,3}y_{e,4}x_{e,4}y_{e,5}w_{g, 2}y_{e,6}x_{e,5}y_{e,7}w_{e,3}$$
and alternate
$$
p_{1}(X, Y, Z) = \sum_{\sigma \in Sym(5)}(-1)^{\sigma}Z_{1, \sigma}.
$$
Now, we evaluate the designated variables $x_{e,i}$ as in the monomial $E_{1}$. It is clear that if we extend the evaluation of the variables in $Y$ and $W$ as in the monomial $E_{1}$ we get a nonzero value ($= e_{1,3}$) of the monomial which correspond to $id \in Sym(5)$ and hence a nonzero value of $p_{1}$. Consider now any permutation $\sigma \in Sym(5)$ which permutes the variables $x_{e,1},\ldots,x_{e,4}$ (and fixes $x_{e,5}$). Evaluating $W$ as above, there exists an evaluation of the set $Y$ such that the monomial $Z_{1,\sigma}$ gets a nonzero value. On the other hand, for any permutation $\tau \in Sym(5)$ which does not fix $x_{e,5}$, there is no value of $Y$ and $W$ which makes the evaluation of the monomial $Z_{1,\tau}$ nonzero.

\begin{rem}
Note that if we allow changing the values of the set $X$, there is no essential difference among the monomials and hence if there is an evaluation of $p_{1}(X, X^{\diamond}, Y, W)$ making the $id$ monomial nonzero, the same can be done for any $Z_{1, \sigma}$.
\end{rem}
We return to the general case where $p_{1}(X, X^{\diamond}, Y, W)$ is the alternating polynomial on the set $X$ constructed above. We know that if we evaluate the monomial $Z_{1}(X, X^{\diamond}, Y, W)$ as in $E_{1}$, we get a nonzero value (by construction), and hence a nonzero value for $p_{1}(X, X^{\diamond}, Y, W)$. Consider the generalized polynomial $p_{1}(\hat{X}, X^{\diamond}, Y, W)$ where the values of $X$ are as in $E_{1}$.
\begin{defn}\label{nonvanishing permutations}
We denote by  $Ad_{Sym(d_{e})}(p_{1}(\hat{X}, X^{\diamond}, Y, W))$ the set of all permutations $\sigma \in Sym(d_{e})$ for which there is an evaluation of the variables $X^{\diamond}$, $Y$ and $W$ such that the monomial $Z_{1, \sigma}(\hat{X}, X^{\diamond}, Y, Z)$ admits a nonzero value.
\end{defn}

Let us pause for a moment and explain what we want to do. With this terminology our goal will be, roughly speaking, to insert successively variables/polynomials $B_{1},\ldots,B_{l}$ in the polynomial $p_{1}(X, X^{\diamond}, Y, W)$ and get polynomials $$p_{(1+i)}(X, X^{\diamond}, Y, W, B_{1},\ldots,B_{i})$$ so that $Ad_{Sym(d_{e})}(p_{(1+i)}(\hat{X}, X^{\diamond}, Y, W, B_{1},\ldots,B_{i}))$ (see Definition \ref{nonvanishing permutations general} below extending \ref{nonvanishing permutations}) has fewer elements but is nonempty (i.e. the polynomial is a nonidentity). Also we want the coefficients  of polynomials $p_{(1+i)}$ and in particular $p_{(1+l)}(X, X^{\diamond}, Y, W, B_{1},\ldots,B_{l})$ to be in $\mathbb{Q}$. Our main objective in that step is to show that permutations $\sigma \in Ad_{Sym(d_{e})}(p_{(1+l)}(\hat{X}, X^{\diamond}, Y, W, B_{1},\ldots,B_{l}))$ preserve the subset of $\Lambda_{e}$ whose elements are represented by $S$, that is $S$-elements can be permuted only to $S$-elements.
The final step will be to insert a polynomial $q_{L}$ defined over $k$, which makes $$p_{(2+l)}(X, X^{\diamond}, Y, W, B_{1},\ldots,B_{l}, q_{L})$$ an identity of $A$. We will show the polynomial identity is defined over $k$ but not over $L$. In order to explain our plan more precisely we introduce some additional terminology.

Our starting point is the polynomial $p_{1}(X, X^{\diamond}, Y, W)$ constructed above and the evaluation of the designated variables $p_{1}(\hat{X}, X^{\diamond}, Y, W)$. As mentioned above, the set $\Lambda_{e}$ is decomposed into diagonal blocks of dimension $d_{i}\times d_{i}$ which correspond to right $H$-coset representatives which appear $d_{i}$-times. By the fixed evaluation $\hat{X}$ we label the variables of $X$ accordingly. We say a variable $x_{i} \in X$ corresponds to the $[k_{a}\times i]$-diagonal block if $\hat{x}_{i}$ belongs to that block.

We fix once for all a nonvanishing evaluation $\hat{X}$ of the set $X$, that is \linebreak $p_{1}(\hat{X}, X^{\diamond}, Y, W)$ is a generalized nonidentity of $A$. Note that by the alternation this can be any of the $d_{e}!$ possible nonvanishing evaluations of $X$. Once we have chosen the evaluation $\hat{X}$ of $X$ we fix a nonvanishing monomial and regard it as the monomial corresponding to the identity permutation in $p_{1}(\hat{X}, X^{\diamond}, Y, W)$. The following definition extends Definition \ref{nonvanishing permutations}. Notation as above.
 \begin{defn}\label{nonvanishing permutations general}
Let $0 \leq i \leq l$. We say a permutation $\sigma \in Sym(d_{e})$ is \textit{nonvanishing} in the polynomial $p_{(1+i)}(\hat{X}, X^{\diamond}, Y, W, B_{1},\ldots,B_{i})$,  if there is an evaluation of $p_{(1+i)}(\hat{X}, X^{\diamond}, Y, W, B_{1},\ldots,B_{i})$ so that the monomial $Z_{1, \sigma}(\hat{X}, X^{\diamond}, Y, Z, B_{1},\ldots,B_{i})$ admits a nonzero value.
\end{defn}

As mentioned earlier there are three type of basis elements in $\Lambda_{e}$, namely (Type $I$) elements that belong to blocks represented by coset representatives of $N_{G}(H)$, (Type $II$) basis elements that belong to blocks represented by coset representatives outside $N_{G}(H)$ and singletons and finally (Type $III$) basis elements that belong to blocks represented by coset representatives outside $N_{G}(H)$ and nonsingletons. By the correspondence established above between the designated variables and the elements of $\Lambda_{e}$, we have designated variables of Type $I$, Type $II$ and Type $III$.

Step $1$:
Consider the polynomial $p_{1}(X, X^{\diamond}, Y, W)$ and insert in each $e$-segment corresponding to an $e$-diagonal block represented by elements in $N_{G}(H)$ a product of variables $x_{h}$, (at least) one for each $h \in H$ and such that the homogeneous degree of the product is $e$ (we may need to add an additional variable). If we denote by $T_{H}$ the set of the inserted $H$-variables and by $p_{2}(X, X^{\diamond}, Y, W, T_{H})$ the obtained polynomial, because diagonal blocks represented by elements in $N_{G}(H)$ have only $H$-pages, the nonzero evaluation of $Z_{1}(\hat{X}, X^{\diamond}, Y, W)$ can be extended to a nonzero evaluation of $Z_{2}(\hat{X}, X^{\diamond}, Y, W, T_{H})$ and hence $p_{2}(X, X^{\diamond}, Y, W, T_{H})$ is a nonidentity.
\begin{prop}
Notation as above. Every nonvanishing permutation of $$p_{2}(\hat{X}, X^{\diamond}, Y, W, T_{H})$$ sends variables of Type $I$ (resp. $II$, $III$) to values of Type $I$ (resp. $II$, $III$).
\end{prop}

\begin{proof}
We show first a nonsingleton variable cannot be evaluated on a singleton element of $\Lambda_{e}$: Indeed, as chosen for nonsingletons, one of the bridges (a variable in $W$) from left or right, are $H$-elements and this is not compatible with a singleton evaluation. In order to prove that singleton variables in $N_{G}(H)$ cannot be evaluated on singletons represented by elements outside $N_{G}(H)$, note that because of the insertion of the $H$-variables in the polynomial, these variables $H$ must be evaluated with elements in $\Lambda$ from the corresponding block. But this is impossible because the evaluation belongs to an element that does not normalize $H$ so not all pages are from $H$, or equivalently, some of the $H$-elements do not belong to the diagonal $e$-block.
\end{proof}
Step $2$: We have shown that $N_{G}(H)$-variables must be evaluated only with basis elements represented by $N_{G}(H)$. Suppose $\sigma \in Sym(d_{e})$ is such nonvanishing permutation. By the proposition, $\sigma$ restricts to permutations of the disjoint sets that correspond to $(a)$ elements that are represented by $N_{G}(H)$ $(b)$ variables represented by representatives not in $N_{G}(H)$ and singletons $(c)$ variables represented by representatives not in $N_{G}(H)$ and nonsingletons. Now, applying Lemma \ref{equal derivative II} and in particular its last part, there is an element in $g_{0} \in N_{G}(H)$ such that $g_{0}\mathfrak{g}$ determines the same $H$-coset representatives as determined by the permutation $\sigma$. Note that in particular, nonvanishing permutations permute the diagonal blocks because by left multiplication by $g_{0}$,  equal representatives are mapped to equal representatives.

Applying Lemma \ref{three equivalent conditions} and Corollary \ref{Definition of the subgroup K} we conclude that $g_{0} \in K$. It follows that the diagonal blocks represented by $K$ correspond to values represented by $K$ and so, in particular, the blocks represented by $S$ are mapped to $K$-blocks. This completes the second step.

Our next step is to insert additional polynomials in $p_{2}(X, X^{\diamond}, Y, W, T_{H})$ so that nonvanishing permutations send $S$-representatives to $S$-values. For that we consider the coset representatives that are in $K \setminus S$ and their action on $M(H)$. Before we do it let us recall in some detail how the binomial identities of a twisted group algebra $F^{\alpha}H$ are related to the Schur Multiplier $M(H)$.

\subsection {}\textit{ Hopf formula for $M(H)$ and binomial identities of $F^{\alpha}H$.}

Let
$$
1 \rightarrow R \rightarrow F_{H} \rightarrow H \rightarrow 1
$$
be a presentation of $H$ where $F_{H}$ is the free group on elements $\{x_{h}: h \in H\}$ and the map $\nu: F_{H} \rightarrow H$ maps $x_{h}$ to $h$. By the Hopf formula we have $M(H) \cong (R \cap [F_{H}, F_{H}])/[R, F_{H}]$. Exchanging elements of $R$ and $F_{H}$ it is not difficult to see that each class in $M(H)$ is represented by an element of the form $$z =(x_{h_1}\cdots x_{h_r})(x_{h_\tau(1)}\cdots x_{h_\tau(r)})^{-1}$$
where $z \in R$ and $\tau \in Sym(r)$. Furthermore, by adding some extra elements of $F_{H}$ on the right of the words $(x_{h_1}\cdots x_{h_r})$ and $(x_{h_\tau(1)}\cdots x_{h_\tau(r)})$ we may assume each of them is in $R$.

Now, let $\alpha \in H^{2}(H, F^{*})$ and let $\pi(\alpha)(z) = \zeta \in F^{*}$, a root of unity. Based on this data we construct an $H$-graded polynomial identity of the twisted group algebra $F^{\alpha}H$. Let $F\langle X_{H} \rangle$ be the free $H$-graded algebra generated by elements $x_{i, h_{i}}, h_{i} \in H, i\in \mathbb{N}$.
\begin{rem}
Our notation of the variables of $X_{H}$ and of the generators of the free group $F_{H}$ is almost the same. This is clearly an abuse of notation especially because on the one hand the elements in $F_{H}$ are invertible and on the other hand the variables in $F\langle X_{H} \rangle$ they are not. Nevertheless, we adopt this notation because the polynomial identities we construct in $F\langle X_{H} \rangle$ are induced by representatives $z \in R \cap [F_{H}, F_{H}]$.
\end{rem}
The following proposition is proved in (\cite{AljHaileNat}, Section $2$).
\begin{prop} \label{binomial identities and Schur multiplier}
Let $\alpha$, $z$, $\zeta$ as above. Then the following binomial
$$
\beta_{\tau}(x_{1, h_1}, \ldots, x_{r, h_r}; \zeta) = x_{1,h_1} \cdots x_{r, h_r}- \zeta\cdot x_{\tau(1), h_\tau(1)}\cdots x_{\tau(r), h_\tau(r)}
$$
is a polynomial identity of $F^{\alpha}H$. Furthermore, the $T$-ideal of identities of $F^{\alpha}H$ is generated by these binomials.
\end{prop}
\begin{rem}
\begin{enumerate}

\item
Note that if we replace $\zeta$ in $\beta_{\tau}(x_{1,h_1}, \ldots, x_{h_r}; \zeta)$ by a scalar $\zeta^{'} \neq \zeta$, then binomial $\beta_{\tau}(x_{1,h_1}, \ldots, x_{h_r}; \zeta^{'})$ is a nonidentity of $F^{\alpha}H$ and moreover its nonzero values in $F^{\alpha}H$ are invertible elements.
\item
In the sequel we will abuse notation even more and denote the variables of $X_{H}$ by $x_{h_i}$ with the implicit assumption that the variables are different even if $h_{i} = h_{j}$ in $H$.
\end{enumerate}
\end{rem}
Step $3$: Let $k_{i}$ a representative in $K \setminus S$. We know it appears $d$-times in $\mathfrak{g}$. Recall that $k_{i}$ does not normalize $B_{\alpha}$ and hence there is a binomial identity of $F^{\alpha}H$
$$\beta_{\rho}(x_{h_1},\ldots,x_{h_m}) = x_{h_{1}}\cdots x_{h_{m}} - x_{h_{\rho{(1)}}}\cdots x_{h_{\rho{(m)}}}$$ some $\rho \in Sym(m)$,
which is not an identity of $F^{\alpha^{k_{i}}}H$. Clearly, by adding one more variable if needed, we may assume the homogeneous degree of $\beta_{\rho}(x_{h_1},\ldots,x_{h_m}) $ is $e$.  We now proceed as in \cite{AljHaile}. Let $R_{d} = R(x_{1},\ldots,x_{d^{2}}, y_{1},\ldots,y_{d^{2}})$ be a Regev polynomial on $2d^{2}$ variables. Recall that this is a nonidentity central polynomial of the algebra of $d \times d$-matrices over $F$. The polynomial is alternating on the $x$ and $y$-variables respectively. We replace all but one variable by homogeneous variables of degree $e$ and a single variable, say $x_{1}$, by an element of degree $h \in H$. We denote the polynomial obtained by $R_{d, h}$. Note that the homogeneous degree of $R_{d, h}$ is $h \in H$. We can now replace each variable $x_{h_{i}}$ in $\beta_{\rho}(x_{h_1},\ldots,x_{h_m})$ by $R_{d, h_{i}}$ and get the polynomial $\beta_{\rho, d} = \beta_{\rho}(R_{d, h_{1}},\ldots,R_{d, h_{m}})$. Note that the homogeneous degree of $\beta_{\rho, d}$ is also $e$ and so can replace (by the $T$-operation) a variables of degree $e$ of $\beta_{\rho, d}$.

In the next lemma we explain how these binomial-central polynomials are used. Insert $\beta_{\rho, d}$ in the segment represented by $k_{i}\in K \setminus S$.
\begin{lem}
The polynomial $p_{3}(X, X^{\diamond}, Y, W, T_{H}, \beta_{\rho, d})$ is a nonidentity of $A$. Furthermore, if $\sigma \in Sym(d_{e})$ is a permutation which sends the variables corresponding to the block represented by $k_{i} \in K \setminus S$ to elements in $\Lambda_{e}$ which correspond to an $S$-block, then $\sigma$ is a vanishing permutation.
\end{lem}
\begin{proof}
The polynomial remains a nonidentity since the conjugation of the binomial $\beta_{\rho}(x_{h_1},\ldots,x_{h_m})$ by $k_{i}$ is a nonidentity of $F^{\alpha}H$ and moreover taking the basic evaluation, the value of the polynomial $\beta_{\rho, d}$ is an invertible element in the $k_{i}$th block. On the other hand, because an element $s \in S$ normalizes $B_{\alpha}$, we have that conjugating the $h$ elements of $\beta_{\rho}(x_{h_1},\ldots,x_{h_m})$ by $s$ yields a polynomial which is again a binomial identity of $F^{\alpha}H$ and hence such evaluation must vanish.
\end{proof}

We complete this step by inserting similar polynomials for all $k_{i} \in K \setminus S$. For simplicity we abuse notation and denote the resulting polynomial again by $p_{3}(X, X^{\diamond}, Y, W, T_{H}, \beta_{\rho, d})$. Note that the binomial we choose depends on $k_{i}$. As a result we obtain that elements of $S$ may be mapped only to $S$-values as long as the permutation is nonvanishing.

The final step, step $4$, is to insert in the first $e$-segment, namely the segment that corresponds to the representative $g_{1} = e$ with $d$ repetitions, a polynomial $q_{L}$ defined over the field $k$ so that the resulting polynomial $p_{4}(X, X^{\diamond}, Y, W, T_{H}, \beta_{\rho, d}, q_{L})$ is an identity of $A$, defined over $k$ but not over $L$.

The general construction in this step is somewhat similar to the one in the previous step.
Recall that by construction, the group $S$ consists of all elements $s \in K$ which normalize $B_{\alpha}$, that is if

$$\beta_{\tau}(x_{h_1},\ldots,x_{h_r}) = x_{h_{1}}\cdots x_{h_{r}} - x_{h_{\tau{(1)}}}\cdots x_{h_{\tau{(r)}}}, \tau \in Sym(r)$$ is an identity of $F^{\alpha}H$ then also
$$\beta_{\tau}(x_{s^{-1}h_1s},\ldots,x_{s^{-1}h_rs}) = x_{s^{-1}h_1s}\cdots x_{s^{-1}h_rs} - x_{s^{-1}h_{\tau{(1)}}s}\cdots x_{s^{-1}h_{\tau{(r)}}s}$$ is an identity of $F^{\alpha}H$.

Let us return to the map $\pi(\alpha) : M(H)\rightarrow F^{*}$ which corresponds to the $2$-cocycle $\alpha \in H^{2}(H, F^{*})$ by means of the Universal Coefficient Theorem. We recall that because $H$ is finite, the Schur multiplier is finite abelian which implies $Im(\pi(\alpha)) = \mu_{n}$, the group of $n$th roots of unity for some $n$. We recall also that because the group $S$ normalizes $B_{\alpha} = ker(\pi(\alpha))$ it acts on the group $\mu_{n}$. We would like to be more precise about how the action of $S$ on $\mu_{n}$ is defined.

Recall that by Proposition \ref{binomial identities and Schur multiplier} the $T$-ideal of identities of $F^{\alpha}H$ is generated by binomial identities and moreover if $\zeta \in \mu_{n}$ where $\mu_{n} = Im(\pi(\alpha)$ there exists a binomial identity of the form
$$\beta_{\tau}(x_{h_1},\ldots,x_{h_r}; \zeta) = x_{h_{1}}\cdots x_{h_{r}} - \zeta x_{h_{\tau{(1)}}}\cdots x_{h_{\tau{(r)}}},$$ for some $r$ and some $\tau \in Sym(r).$
Now, if $u_{h_1}$ and $u_{h_2}$ are representatives of $h_{1}$ and $h_{2}$ in $F^{\alpha}H$, applying the formula $u_{h_i}u_{h_j} = \alpha(h_{i}, h_{j})u_{h_ih_{j}}$ several times we have $u_{h_1}\cdots u_{h_r} = \alpha(h_{1},\ldots,h_{r})u_{h_{1}\cdots h_{r}}$ and from the fact that $\beta_{\tau}(x_{h_1},\ldots,x_{h_r}; \zeta)$ is a binomial identity of $F^{\alpha}H$ we have that
$$
\zeta = \alpha(h_{1},\ldots,h_{r})(\alpha(h_{\tau{(1)}},\ldots,h_{\tau{(r)}}))^{-1}.
$$
Let $s \in S$. Using this expression we set
$$
s(\zeta) = \alpha(s^{-1}h_{1}s,\ldots,s^{-1}h_{r}s)\alpha(s^{-1}h_{\tau{(1)}}s,\ldots,s^{-1}h_{\tau{(r)}}s)^{-1}.
$$
The fact that $S$ normalizes $B_{\alpha}$ says that the above action on $\mu_{n}$ is well defined, that is, it does not depend on the choice of the binomial identity. Recalling our notation $k = \mathbb{Q}(\mu_{n})^{S}$ we have a well defined surjective map
$$
S \rightarrow Gal(\mathbb{Q}(\mu_{n})/k) = \bar{S}
$$
and composing with the embedding of $\bar{S} \leq Gal(\mathbb{Q}(\mu_{n})/\mathbb{Q})$ we obtain a homomorphism
$$
\phi: S \rightarrow Gal(\mathbb{Q}(\mu_{n})/\mathbb{Q}).
$$

Next, we shall consider the action of the groups $N_{G}(H)$ and $Gal(\mathbb{Q}(\mu_{n})/\mathbb{Q})$ on $H^{2}(H, \mathbb{Q}(\mu_{n})^{*})$ and the induced action on the set of  multilinear binomials

$$ \{\beta_{\tau}(x_{h_1},\ldots,x_{h_r}; \zeta) = x_{h_{1}}\cdots x_{h_{r}} - \zeta x_{h_{\tau{(1)}}}\cdots x_{h_{\tau{(r)}}}\}_{r >1, \tau \in Sym(r), \zeta \in \mu_{n}}.$$

If $g \in N_{G}(H)$ we let $\alpha^{g}: H \times H \rightarrow F^{*}$ be the $2$-cocycle defined by $\alpha^{(g)}(h_1, h_2) = \alpha(g^{-1}h_{1}g, g^{-1}h_{2}g)$ whereas for $g \in Gal(\mathbb{Q}(\mu_{n})/\mathbb{Q})$ we let $\alpha^{[g]}(h_1, h_2) = g(\alpha(h_1, h_2))$. Note that with this notation the actions on $H^{2}(H, \mathbb{Q}(\mu_{n})^{*})$ coincide upon restriction to $S$ and so we write $\alpha^{s} = \alpha^{(s)} = \alpha^{[s]}$

\begin{lem} \label{different actions on binomials} The following hold for every element $s \in S$.

\begin{enumerate}
\item
For every binomial identity of $F^{\alpha}H$
$$\beta_{\tau}(x_{h_1},\ldots,x_{h_r}; \zeta) = x_{h_{1}}\cdots x_{h_{r}} - \zeta x_{h_{\tau{(1)}}}\cdots x_{h_{\tau{(r)}}}$$ where $\zeta \in \mu_{n} = Im(\pi(\alpha))$, $r > 1$ and $\tau \in Sym(r)$, the polynomial
$$\beta_{\tau}(x_{s^{-1}h_1s},\ldots,x_{s^{-1}h_{r}s}; s(\zeta)) = x_{s^{-1}h_{1}s}\cdots x_{s^{-1}h_{r}s} - s(\zeta) x_{s^{-1}h_{\tau{(1)}}s}\cdots x_{s^{-1}h_{\tau{(r)}}s}$$ is an identity of $F^{\alpha}H$.

\item
For every binomial identity of $F^{\alpha}H$
$$\beta_{\tau}(x_{h_1},\ldots,x_{h_r}; \zeta) = x_{h_{1}}\cdots x_{h_{r}} - \zeta x_{h_{\tau{(1)}}}\cdots x_{h_{\tau{(r)}}}$$ where $\zeta \in \mu_{n} = Im(\pi(\alpha))$, $r > 1$ and $\tau \in Sym(r)$, the polynomial
$$\beta_{\tau}(x_{h_1},\ldots,x_{h_r}, s(\zeta))$$ is an identity of $F^{\alpha^{s}}H$.

\item
For every binomial identity of $F^{\alpha}H$
$$\beta_{\tau}(x_{h_1},\ldots,x_{h_r}; \zeta) = x_{h_{1}}\cdots x_{h_{r}} - \zeta x_{h_{\tau{(1)}}}\cdots x_{h_{\tau{(r)}}}$$ where $\zeta \in \mu_{n} = Im(\pi(\alpha))$, $r > 1$ and $\tau \in Sym(r)$, the polynomial
$$\beta_{\tau}(x_{s^{-1}h_1s},\ldots,x_{s^{-1}h_rs}; \zeta)$$ is an identity of $F^{\alpha^{s^{-1}}}H$.

\end{enumerate}
Furthermore, let $g \in Gal(\mathbb{Q}(\mu_{n})/\mathbb{Q}) \setminus \bar{S}$, and let $\zeta_{n} \in \mu_{n}$ be a primitive $n$th root of unity. If $\beta_{\tau}(x_{h_1},\ldots,x_{h_r}; \zeta_{n})$ is a binomial identity of $F^{\alpha}H$ for some $r > 1$ and some $\tau \in Sym(r)$, then $$\beta_{\tau \in Sym(r)}(x_{h_1},\ldots,x_{h_r}, g(\zeta_{n}))$$ is a nonidentity of $F^{\alpha^{s}}H$, for every $s \in S$.

\end{lem}
\begin{proof}
The proofs of $1-3$ follow from the definition of $S$. The proof of the last part follows from $2$ and the fact that $g(\zeta_{n}) \neq s(\zeta_{n})$ for every $s \in S$.
\end{proof}

Let $$\beta_{\tau}(x_{h_1},\ldots,x_{h_r}; \zeta_{n}) = x_{h_{1}}\cdots x_{h_{r}} - \zeta_{n}x_{h_{\tau{(1)}}}\cdots x_{h_{\tau{(r)}}}$$
be an identity of $F^{\alpha}H$ where $\zeta_{n}$ is a primitive $n$th root of unity and consider the product
$$
m(X) = \Pi_{s \in S}\beta_{\tau}(x_{h_1},\ldots,x_{h_r}; s(\zeta_{n})).
$$
By the lemma, the product is an identity of $F^{\alpha^{s}}H$ for every $s \in S$ and is clearly defined over $k = \mathbb{Q}(\zeta_{n})^{S}$. As usual we assume the polynomial is multilinear. Our next goal is to show the following
\begin{prop} \label{Definition over L}
The identity $m(X)$ is not defined over $L$.
\end{prop}
Once we know that, we shall proceed with the last step, that is, given a field $L$ which does not contain $k$, construct an identity of $A$ which is not defined over $L$. For the proof of the proposition we need a general lemma.

Let $C$ be a $G$-graded simple algebra over an algebraically closed field of characteristic zero. Let $P_{C} = (H, \gamma, \mathfrak{g})$ be a presentation of $C$. Denote by $\eta_{\gamma}: M(H) \rightarrow F^{*}$ the map defined on the Schur multiplier $M(H)$ which corresponds to $\gamma$ by means of the Universal Coefficient Theorem. Let $Im(\eta_{\gamma}) = \mu_{n}$ be the group of all $n$th roots of unity. Clearly we may assume the cocycle $\gamma$ has values in $F^{*}_{0}$ where $F_{0} = \mathbb{Q}(\mu_{n})$. It follows that we may consider the twisted group algebra $F_{0}^{\gamma}H$ and hence also the corresponding $G$-graded simple algebra $C_{\gamma}$ over $F_{0}$. Let $L$ be a subfield of $F_{0}$ and let $U= Gal(F_{0}/L)$ be the corresponding Galois group. Consider the twisted group algebras $F_{0}^{\alpha^{g}}H$ with cohomology classes $[\alpha^{g}] \in H^{2}(H, F_{0}^{*})$, $g \in U$, and let $C_{\gamma^{g}}$ be the corresponding $G$-graded simple algebra. Note that $[\gamma^{g}]$ is the cohomology class that corresponds to $g\circ \eta: M(H) \rightarrow F^{*}$. Let $p(X)$ be an identity of $C_{\gamma}$ and suppose $p$ is defined over $L$, that is $p(X) \in F \otimes_{L}Id_{L}(C_{\gamma})$. Explicitly, there exist identities $p_{i} \in Id_{L}(C_{\gamma})$ and scalars $b_{i}$ such that

$$
p = \sum_{i}b_{i}p_{i}.
$$

\begin{lem} \label{binomial identities}
Notation as above. The polynomial $p(X)$ is an identity of $C_{\gamma^{g}}$, for every $g \in U$.
\end{lem}

\begin{proof}
It is sufficient to show that each $p_{i}$ is an identity of $C_{\gamma^{g}}$. Replacing $p_{i}$ by $p$ we let
$$
p = \sum_{\sigma}\beta_{\sigma}m_{\sigma}
$$
where $\beta_{\sigma} \in L$ and $m_{\sigma}$ are monomials.
Moreover we can assume $p$ is multilinear and so every evaluation of the variables on basis elements $u_{h}\otimes e_{i,j} \in C_{\gamma^{g}}$ yields the same value as the corresponding evaluation on $C_{\gamma}$ where a scalar $\lambda_{\sigma} \in \mu_{n}$ is replaced by $\lambda^{g}_{\sigma}$. Because $p$ is an identity of $C_{\gamma}$ we have $\sum_{\sigma}\beta_{\sigma}\lambda_{\sigma}T_{\sigma} = 0$ where we may assume $T_{\sigma} = 1 \otimes e_{1,1}$ or zero.
Now, if we evaluate the polynomial $p$ on $C_{\gamma^{g}}$, we obtain
$$
\sum_{\sigma}\beta_{\sigma}\lambda^{g}_{\sigma}T_{\sigma} = \sum_{\sigma}(\beta_{\sigma}\lambda_{\sigma})^{g}T_{\sigma} = 0
$$
because the $\beta's$ are $U$ invariant.

\end{proof}

We proceed now with the proof of Proposition \ref{Definition over L}.

\begin{proof}
Suppose the converse holds. By Lemma \ref{binomial identities} and the paragraph preceding it, replacing $L$ in the lemma by the field $L\cap \mathbb{Q}(\mu_{n})$, the polynomial $m(X)$ is an identity of $F^{\alpha^{g}}H$ for every element $g \in Gal(\mathbb{Q}(\mu_{n})/(L\cap \mathbb{Q}(\mu_{n})))$ where here the action of $Gal(\mathbb{Q}(\mu_{n})/(L\cap \mathbb{Q}(\mu_{n})))$ on $\alpha$ is induced by the action on its values, namely on $\mu_{n}$.

Now if $g \in Gal(\mathbb{Q}(\mu_{n})/(L\cap \mathbb{Q}(\mu_{n}))) \setminus \bar{S}$, the value $g(\zeta_{n}) \neq s(\zeta_{n})$ for every $s \in S$ and hence invoking the last part of Lemma \ref{different actions on binomials} each binomial factor in $m(X)$ is a nonidentity of $F^{\alpha^{g}}H$. We claim also the product
$$
m(X) = \Pi_{s \in S}\beta_{\tau}(x_{h_1},\ldots,x_{h_r}; s(\zeta_{n}))$$
is a nonidentity of $F^{\alpha^{g}}H$. Indeed, every nonzero value of $\beta_{\tau}(x_{h_1},\ldots,x_{h_r}; s(\zeta_{n}))$ is invertible in $F^{\alpha^{g}}H$ and the claim follows.
\end{proof}
We would like to use the polynomial $m(X)$ for the construction of an identity of $A$, defined over $k$ but not over $L$.
We replace every variable of $m(X)$ by a central polynomial $R_{d, h_{i}}$ and denote the polynomial by $q_{L}$. We insert the polynomial $q_{L}$ in the first $e$-segment of $p_{3}(X, X^{\diamond}, Y, W, T_{H}, \beta_{\rho, d})$, namely the segment which corresponds to the diagonal block $d \times d$ represented by the trivial $H$-coset representative $e$. We denote the resulting polynomial by $p_{4}(X, X^{\diamond}, Y, W, T_{H}, \beta_{\rho, d}, q_{L})$.
\begin{prop}
The polynomial $p_{4}(X, X^{\diamond}, Y, W, T_{H}, \beta_{\rho, d}, q_{L})$ is an identity of $A$.
\end{prop}
\begin{proof}
Let us show the generalized polynomial $p_{4}(\hat{X}, X^{\diamond}, Y, W, T_{H}, \beta_{\rho, d}, q_{L})$ has no nonvanishing permutations. Indeed, the variables represented by $S$ can be permuted in any nonvanishing permutation only to basis elements represented by $S$. The presence of $q_{L}$ annihilates the evaluation also in that case.
\end{proof}
\begin{prop}
The identity $p_{4}(X, X^{\diamond}, Y, W, T_{H}, \beta_{\rho, d}, q_{L})$ is not defined over $L$, whenever $L\cap \mathbb{Q}(\mu_{n}) \nsupseteq k$.
\end{prop}
\begin{proof}
We consider an action of $Gal(\mathbb{Q}(\mu_{n})/\mathbb{Q})$ and in particular of $Gal(\mathbb{Q}(\mu_{n})/(L\cap \mathbb{Q}(\mu_{n})))$ on the family of $G$-graded simple algebras via the action on the $2$-cocycle $\alpha$ and this, via the action on its values. We denote by $A_{\alpha^{g}}$ the image of $A = A_{\alpha}$ induced by $g \in Gal(\mathbb{Q}(\mu_{n})/(L\cap \mathbb{Q}(\mu_{n})))$. The polynomial $p_{3}$ above, obtained prior to the insertion of $q_{L}$, is defined over $\mathbb{Q}$ and hence it is not an identity of $A_{\alpha^{g}}$. The insertion of $q_{L}$ produces an identity of $A$ and if defined over $L$ it is an identity of $A_{\alpha^{g}}$. Let us show this is false: The argument is similar to the proof above, namely that $m(X)$ is a nonidentity of the twisted group algebra $F^{\alpha^{g}}H$. Indeed, because the variables of the binomial factors of $m(X)$ are replaced by central polynomials of $d\times d$-matrices, the values of the corresponding factors in $q_{L}$ obtain nonzero central values and hence invertible, within the corresponding $e$-block. This shows their product is nonzero and hence if the polynomial $p_{3}$ before the insertion of $q_{L}$ had a monomial with nonzero value on $A_{\alpha^{g}}$, the same monomial will not vanish upon the insertion of $q_{L}$. This proves the proposition. The proof of Theorem \ref{field of definition theorem} is now complete.
\end{proof}

As mentioned in Remark \ref{Ehud Meir Example}, there exist PI algebras $B$ with minimal field of definition that strictly contains the minimal field of definition of $Id(B)$. We close this section by showing that under some mild conditions the $T$-ideal of identities of an algebra $B$ has a unique minimal field of definition. The proof was communicated to us by Ehud Meir.
\begin{prop}
Let $B$ be an associative finite dimensional algebra over $F$. Suppose $I = Id(B) \leq k\langle X \rangle$ is its $T$-ideal of identities defined over an algebraic and Galois extension $k$ of $\mathbb{Q}$. Then there exists a subfield $k_{min}$ of $k$ such that the following holds.
\begin{enumerate}
\item
$k_{min}$ is a field of definition of $Id(B)$.
\item
If $L \leq k$ is a field of definition of $B$ then $k_{min} \leq L$.
\item
The form of $Id(B)$ over $k_{min}$ is unique.

\end{enumerate}

\end{prop}
\begin{proof}

We note that the algebra $k\langle X\rangle$ is already defined over $\mathbb{Q}$ and is graded by $\mathbb{N}^{+}$ where the grading is determined by the length of the monomials.
It is sufficient to show that $I_{n}$, the $n$th homogeneous component of $I$, has a unique minimal field of definition $k_{n}$ and over $k_{n}$ $I$ has a unique form, for then $k_{min} = \sum k_n$ will be the unique field of definition of $I = Id(B)$ and $I = \oplus I_{n}$ the unique form.

In order to prove that we recall from (\cite{AlBeKar}, Theorem $8.2$) that we may assume $X$ is finite which implies $I_{n}$ is finite dimensional vector space over $k$.

Write $\Lambda = \text{Gal}(k/\mathbb{Q})$. The group $\Lambda$ acts on $k\langle X\rangle_n  = k\otimes_{\mathbb{Q}} \mathbb{Q}\langle X\rangle_n$ by its action on $k$. It thus also acts on the set of subspaces of $k\langle X\rangle_n$. Write $H_n = \text{stab}_\Lambda(I_n)$.
Claim: $I_n$ has a unique form over the unique minimal field of definition $k_n = k^{H_n}$.
Indeed, if $I_n$ is defined over $k^H$ for some $H<\Lambda$ then $I_n = V\otimes{_{k^H}} k$, where $V$ is a subspace of $k^H\langle X\rangle$. It is then easy to see that every $h\in H$ stabilizes $I_n$ as a subspace. This means that $H<H_n$ and therefore $k_n\subseteq k^H$. On the other hand, by a result due to Speiser (\cite{GilleSzamuely}, Lemma 2.3.8.) $I_{n}^{H_{n}}$ is a form of $I_{n}$ over $k_{n}$ and hence $k_{n}$ is its unique minimal field of definition. Let us show $I_{n}^{H_{n}}$ is the only form of $I_{n}$ over $k_{n}$. Fix a basis $v_1,\ldots v_d$ for $I_n$. Then for every $h\in H_n$ we have that $h(v_j) = \sum_i a_{ij}(h) v_i$ for some uniquely defined $a_{ij}(h)\in k$. The assignment $h\mapsto a_{ij}(h)$ then gives a one cocycle which represents a class in the non-abelian cohomology $H^1(H_n,GL_d(k))$. By the general version of Hilbert theorem $90$, namely for the group $U = GL_{d}(k)$, we have that the cohomology set $H^1(H_n,GL_d(k))$ is trivial. But by descent theory the set of forms of $I_{n}$ over $k_{n}$ is in bijection with the set $H^1(H_n,GL_d(k))$ and the claim follows.
\end{proof}

\section{generic algebras}

In this section we construct the algebra~ $\mathcal{U}$ and prove the Artin-Procesi's condition (Theorem \ref{main theorem}, $1-3$).

We fix a finite dimensional $G$-graded simple algebra $A$ over an algebraically closed field $F$ of characteristic zero. Let $k$ be the unique minimal field of definition of $A$ and let $k\langle X_{G} \rangle$ be the free $G$-graded algebra over $k$. As shown in the previous section, the field $k$ is a field of definition (in fact unique minimal) also for the $T$-ideal of $G$-graded identities of $A$ and hence there is $\Gamma = Id_{G,k}(A)$ with $\Gamma \otimes_{k}F = Id_{G,F}(A)$. This allows us to consider the relatively free algebra $\mathcal{A} = k\langle X_{G} \rangle/\Gamma$.

Next we recall that such an algebra $A$, namely finite dimensional $G$-graded simple, admits a nonidentity $G$-graded $e$-central polynomial $f$ (see \cite{Karasik}). Moreover, by its construction, $f$ is alternating on homogeneous sets of variables $X_{g}$ of cardinality $d_{g}$, $g \in G$, where $d_{g}$ is the dimension of $A_{g}$ as a vector space over $F$. Now, since $A$ is defined over $k$, writing $f =  \sum_{i}\beta_{i}p_{i}$ where the $p_{i}$'s are polynomials with coefficients in $k$ and the $\beta_{i}'s \in F^{*}$ are linearly independent over $k$, we have that every $p_{i}$ is $e$-central and at least one of them, say $p_{1}$, is a nonidentity. Note that by decomposing $f$ over $k$ we may have lost the alternation property of $f$ and so in order to retrieve it back, we insert in $p_{1}$ a Kemer polynomial of $A$, defined over $\mathbb{Q}$ as constructed in \cite{AB}. We obtain again an $e$-central polynomial of $A$ with coefficients in $k$ alternating on homogeneous sets of variables $X_{g}$ of cardinality $d_{g}$. We denote this polynomial again by $f$.

Let us briefly outline the proof of Theorem \ref{main theorem} (the detailed proof is contained in this section and in the next two sections): We would like to localize $\mathcal{A}$ on a central polynomial $m(X)$ which turns out to be a \textit{a refinement} of $f$, so that the localized algebra $m(X)^{-1}\mathcal{A}$ satisfies the polynomial identities of $A$ and no nonzero homomorphic image of $m(X)^{-1}\mathcal{A}$ satisfies more identities than $A$ (that is, $m(X)^{-1}\mathcal{A}$ is Azumaya in the sense of Artin-Procesi). This will be the $\mathcal{U}$ of the theorem. For the localization we will need to show that the nonzero elements of the $e$-center of $\mathcal{A} = k\langle X_{G} \rangle/\Gamma$ are nonzero divisors in $\mathcal{A} = k\langle X_{G} \rangle/\Gamma$. Note that this implies in particular that $R = m(X)^{-1}Z(\mathcal{A})_{e}$ is an integral domain (see \cite{Karasik}).

After proving the algebra $\mathcal{U} = m(X)^{-1}\mathcal{A}$ is Azumaya in the sense of Artin-Procesi, we show that the $G$-graded two sided ideals of $\mathcal{U}$ are in $1-1$ correspondence with the ideals of $R$, the $e$-center of $\mathcal{U}$ (part $4$ of the theorem). Azumaya in the sense of Artin-Procesi $+$ the $1-1$ correspondence between $G$-graded ideals and $e$-central ideals of $\mathcal{U}$ allows us to prove the algebra $\mathcal{U}$ is representing (part $5$ of the theorem).

We start with the construction of the polynomial $m(X)$, a refinement of $f$.

\begin{prop}\label{vanishing e-central polynomials on E}
There exists an $e$-central polynomial $m(X)$ of $A$, defined over $k$, which is an identity of every $G$-graded simple algebra $E$ with $Id_{G}(E) \varsupsetneq Id_{G}(A)$.
\end{prop}
\begin{proof} For an algebra $E$ as above we have the following.

\begin{lem}
There is an homogeneous polynomial $z_{E}$, defined over $\mathbb{Q}$, which is an identity of $E$ but not of $A$.
\end{lem}
\begin{proof}
The proof of the lemma strongly relies on proof of the main Theorem in \cite{AljHaile} (see \textit{Theorem} in the introduction) which says that finite dimensional $G$-graded simple algebras are determined by their $T$-ideal of identities. Suppose $A$ and $B$ are two arbitrary finite dimensional $G$-graded simple algebras. Let $P_{A}$ and $P_{B}$ be presentations in the sense of Bahturin, Sehgal and Zaicev of $A$ and $B$ respectively. The proof of the theorem in \cite{AljHaile} consists of $10$ steps in which the authors construct nonidentities $f$ of one algebra, either $A$ or $B$, (in each step it may be a different one, depending on their possible structure) and use the assumption that $f$ is also a nonidentity of the other algebra in order to put constrains on its $G$-graded structure (in comparison to the $G$-graded structure of the former one). In general, the polynomials needed for the proof in  \cite{AljHaile} are not defined over $\mathbb{Q}$ but over a cyclotomic extension of $\mathbb{Q}$. Nevertheless, a key point in the proof is that if $A$ and $B$ share the same identities defined over $\mathbb{Q}$ (and hence the same $\mathbb{Q}$-nonidentities) this is sufficient to impose either $A \cong B$ or $Id_{G}(A) \nsubseteq Id_{G}(B)$ and $Id_{G}(B) \nsubseteq Id_{G}(A)$. Since we are assuming $Id_{G}(E) \varsupsetneq Id_{G}(A)$ the lemma follows.
For the reader convenience, let us provide some details.

We fix $P_{A} = (H, \mathfrak{g}, \alpha)$ and $P_{E} = (\hat{H}, \hat{\mathfrak{g}}, \hat{\alpha})$, presentations of $A$ and $E$ respectively. Recall that any $G$-graded simple algebra is \textit{basic} (or \textit{fundamental}) (see \cite{AB}) and hence it admits Kemer polynomials which alternate on sets of variables of degree $g$ of cardinality $d_{g}$, where $d_{g}$ is the dimension of the $g$-component. It follows immediately that $dim(E_{g})$ is bounded by $dim(A_{g})$ for otherwise a Kemer polynomial for $E$ would be an identity of $A$ contradicting our assumption. On the other hand, if $dim(E_{g}) < dim(A_{g})$ for some $g \in G$, because Kemer polynomials are defined over $\mathbb{Q}$, we may take $z_{A}$ $=$ any Kemer polynomial of $A$, which clearly vanishes on $E$. Thus for the proof of the lemma, it is sufficient to consider the case where $dim_{F}(E_{g}) = dim_{F}(A_{g})$, every $g \in G$.

\begin{rem}
It is known that if $G$ is abelian, the conditions above, namely
\begin{enumerate}
\item
$A$ and $E$ are finite dimensional $G$-graded simple

\item
$dim_{F}(E_{g}) = dim_{F}(A_{g})$, every $g \in G$
\item
$Id_{G}(A) \subseteqq Id_{G}(E)$

\end{enumerate}
already implies that $A$ and $E$ are $G$-graded isomorphic(!) and hence the lemma is proved in that case (see \cite{OfirDavid}). Interestingly, for arbitrary $G$ this is not known.
\end{rem}
Our next step is to show there exists $z_{E} \in Id_{G}(E) \setminus Id_{G}(A)$, defined over $\mathbb{Q}$, unless the $e$-blocks determined by $\mathfrak{g}$ and $\hat{\mathfrak{g}}$ have the same sizes and in particular $\mathfrak{g}$ and $\hat{\mathfrak{g}}$ have the same cardinality. Note that this will imply that $H$ and $\hat{H}$ have the same order. To see this, we order the tuples $\mathfrak{g}$ and $\hat{\mathfrak{g}}$ as in the proof following Lemma \ref{equal derivative II} and let $d_{1}, \ldots, d_{r}$ and $\hat{d}_{1},\ldots,\hat{d}_{s}$ be the corresponding multiplicities. We are assuming $d_{1} \geq d_{2} \cdots \geq d_{r} \geq 1$ and $\hat{d}_{1} \geq \hat{d}_{2} \cdots \geq \hat{d}_{s} \geq 1$. Note that the cardinality of $\mathfrak{g}$ (resp. $\hat{\mathfrak{g}}$) is $\sum_{i}d_{i}$ (resp. $\sum_{j}\hat{d}_{j}$). Note also that $\delta_{A} = (d^{2}_{1}, \ldots, d^{2}_{r})$ and $\delta_{E} = (\hat{d}^{2}_{1},\ldots,\hat{d}^{2}_{s})$ are partitions of $dim_{F}(A_{e}) = dim_{F}(E_{e})$. We want to show a suitable $z_{E}$ exists unless the partitions are equal.
Suppose first $\delta_{A} \precneqq \delta_{E}$ and construct a monomial for the algebra $E$ with designated variables, frames and bridges as in \cite{AljHaile}. Alternating the designated variables we obtain a nonidentity $p$ of $E$. Moreover, in any nonzero evaluation, because of the alternation, we must evaluate all designated variables on the entire $e$-component of $E$ or $A$. We claim $p$ is an identity of $A$ contradicting the assumption $Id_{G}(A) \subseteqq Id_{G}(E)$. Indeed, there is no nonvanishing evaluation of $p$ on $A$ as long as an $e$-segment of $p$ is evaluated on $e$-elements which belong to different $e$-blocks which must be the case if $\delta_{E} \succneqq \delta_{A}$.
A similar argument can be used also in case $\delta_{A} \succneqq \delta_{E}$. Indeed, as above, we may construct a nonidentity $z_{E}$ of $A$, defined over $\mathbb{Q}$, which is an identity of $E$.

In the next steps, as they appear in \cite{AljHaile}, we show a suitable $z_{E} \in Id_{G}(E) \setminus Id_{G}(A)$ (i.e. defined over $\mathbb{Q}$) exists unless we have $H = \hat{H}$ and $\mathfrak{g} = \hat{\mathfrak{g}}$ (we may need to apply basic moves on the presentations of $A$ or $E$ to achieve that).

Thus, we have found a suitable $z_{E}$ unless the presentations of $A$ and $E$ are equal or else they differ by the cohomology class. Clearly, the presentations cannot be equal if $Id_{G}(A) \varsubsetneq Id_{G}(E)$. The final step in the proof shows that if the presentations differ only in the cohomology class, then again $Id_{G}(A) \varsubsetneq Id_{G}(E)$ is not possible. Details are omitted.

\end{proof}
Now, for the construction of $m(X)$ we take the central polynomial $f$, defined over $k$, and replace one of the designated $e$-variables by $z_{E}$ (note that we may assume by adding a suitable variable, that the homogeneous degree of $z_{E}$ is $e$). We denote the resulting polynomial by $\hat{f}$. Thus the polynomial $\hat{f}$ is a nonidentity $e$-central polynomial of $A$ which is defined over $k$. Moreover, $\hat{f}$ is an identity of $E$. Next we want to do the same for every $G$-graded simple $E$ with $Id_{G}(A) \varsubsetneq Id_{G}(E)$.

\begin{lem}
There are only a finite number of finite dimensional $G$-graded simple algebras with $Id_{G}(A) \varsubsetneq Id_{G}(E)$ over $F = \bar{F}$.
\end{lem}

\begin{proof}
We have seen that $dim(E_{g}) \leq dim(A_{g})$. This shows the order of the group $\hat{H}$ and the cardinality of the tuple $\hat{\mathfrak{g}}$ are bounded. Finally, because $H$ is finite, the Schur multiplier $M(H)$ is a finite (abelian) group and hence, because the number of roots of unity in $F$ of order dividing $exp(M(H))$ is finite, the group $Hom(M(H), F^{*})$ is finite and the result follows.
\end{proof}

Thus, in order to complete the construction of $m(X)$ we insert in $f$ polynomials $z_{E}$ for every $G$-graded simple algebra $E$ with $Id_{G}(A) \varsubsetneq Id_{G}(E)$. This completes the proof of Proposition \ref{vanishing e-central polynomials on E}.
\end{proof}

Our conclusion at this point is that the algebra $m^{-1}\mathcal{A}$ satisfies precisely the identities of $A$ and there is no nonzero homomorphic image $B$ with $Id_{G}(B) \supseteq Id_{G}(E) \varsupsetneq Id_{G}(A)$ where $E$ is $G$-graded simple. Indeed, the polynomial $m(X)$ would vanish on $E$ and hence on $B$ which is impossible if $m(X)$ is invertible.

We can now complete the proof of part $3$ in Theorem \ref{main theorem}. Suppose $Id_{G}(B) \varsupsetneq Id_{G}(A)$. Because $A$ is finite dimensional the $T$-ideal $Id_{G}(A)$ and hence also $Id_{G}(B)$ contains a nongraded Capelli polynomial and so we may apply Kemer's representability theorem for $G$-graded algebras ($G$-finite). Consequently we obtain a finite dimensional algebra $B_{0}$ over $F$ which is PI equivalent to $B$. Let $B_{0} \cong E_{1} \times \cdots \times E_{q} \oplus J(B_{0})$ be the Wedderburn decomposition of $B_{0}$ as a $G$-graded algebra where the $E_{i}$, $i = 1,\ldots, q$, are $G$-graded simple algebras and $J(B_{0})$ is the radical Jacobson of $B_{0}$. Here, the sum is direct as vector spaces over $F$. But $Id_{G}(E_{i}) \supseteq Id_{G}(B_{0}) \supsetneq Id_{G}(A)$ and so, applying Proposition \ref{vanishing e-central polynomials on E}, the polynomial $m(X)$ vanishes on $E_{i}$, $i=1,\ldots,q$. It follows that $m(X)$ is an identity of the semisimple supplement of $J(B_{0})$ in $B_{0}$ and hence every evaluation of $m(X)$ on $B_{0}$ is contained in $J(B_{0})$. Denoting by $n_{B_{0}}$ the nilpotency index of $J(B_{0})$ we have that $m(X)^{n_{B_{0}}}$ is an identity of $B_{0}$ and by the PI equivalence also of $B$. This contradicts the fact that $m(X)$ is invertible in $\mathcal{U}$ and hence invertible also in its image $B$.

\section{Graded Azumaya algebras}

Let $A$ be a $G$-graded $F$-algebra, where $F$ is a field of characteristic
$0$. Let $R=Z(A)_e$ be the e-center of $A$.
Note that $R$ might not be a field.

The goal of this section is to show the following Theorem:
\begin{thm}
\label{idealCorrespondence}Suppose that $A$ also satisfies:
\begin{enumerate}
\item $A$ is Azumaya $($as an ungraded algebra$)$.
\item $J\left(R\right)=0$.
\end{enumerate}
Then, there is a $1-1$ correspondence between ideals
of $R$ and $G$-graded ideals of $A$ given by
\begin{align*}
A & \vartriangleright I\mapsto I\cap R\vartriangleleft R\\
 & R\vartriangleright\mathfrak{a}\mapsto\mathfrak{a}A\vartriangleleft A.
\end{align*}
Furthermore, condition (2) is satisfied if $A$ satisfies (1), is PI and $J(A)=0$.
\end{thm}

We will assume till the end of this section that $A$ is a $G$-graded
algebra which is also an Azumaya algebra.
\begin{lem} \label{e_center_preservation} If $B$ is another $G$-graded algebra
over $R$ and $\phi: A\to B$ is a graded $R$-epimorphism, then $B$
is Azumaya and
\[
Z_{e}\left(B\right)=\phi\left(Z_{e}\left(A\right)\right).
\]
\end{lem}

\begin{proof}
Because $A$ is separable and $\phi$ is an epimorphism, by Proposition
1.11 in \cite{Braun}, $B$ is Azumaya and $Z\left(B\right)=\phi\left(Z\left(A\right)\right)$.
As $\phi$ preserves the grading, we deduce:
\[
Z_{e}\left(B\right)=\phi\left(Z\left(A\right)\right)_{e}=\phi\left(Z_{e}\left(A\right)\right).
\]
\end{proof}
\begin{lem} \label{azumaya_over_field}If $R$ is a field, then $A$ is a
finite dimensional $G$-graded simple $R$-algebra.
\end{lem}

\begin{proof}
By \cite[Proposition 1.12]{Braun}, $A$ is separable over $R$. By \cite[Theorem 2.5]{Braun},
$A$ is finite dimensional over $R$ and is semisimple as an ungraded
algebra. As a result, $A$ is $G$-semisimple. Because $R$ is a field,
$A$ must be $G$-graded simple.
\end{proof}
\begin{cor}
\label{cor:maximal_ideals_correspondence}If $M$ is a $G$-graded
maximal ideal of $A$, then $M=\mathfrak{m}A$, where $\mathfrak{m}=M\cap F$.
\end{cor}

\begin{proof}
Clearly $M\supseteq\mathfrak{m}A$. By Lemma \ref{e_center_preservation}, the algebra $A/\mathfrak{m}A$ has an $e$-center which is a field and
is Azumaya by \cite[Proposition 1.11]{Braun}. Hence, by Lemma \ref{azumaya_over_field},
it is $G$-graded simple. Furthermore, this algebra projects onto $A/M$
and hence the projection is a $G$-isomorphism. We conclude that $M=\mathfrak{m}A$.
\end{proof}
\begin{rem}
It is clear from the proof that if $\mathfrak{m}$ is a maximal ideal
of $R$, then so does $M=\mathfrak{m}A$.
\end{rem}

We now can conclude the ``furthermore'' part of the Theorem. Indeed, in the case where $A$ is a PI algebra, $J\left(A\right)$
is equal to the $G$-graded Jacobson radical (\cite[Theorem 4.4]{Dicks})
and by the $G$-graded Kaplansky Theorem (\cite[Corollary 2.18]{Karasik}),
it is the intersection of two-sided maximal graded ideals. By the above Corollary, it is equal to $J(R)A$. As a result,  $J(A)=0$ if and only if $J(R)=0$.

\begin{lem}
$R$ is a direct $R$-summand of $Z\left(A\right)$.
\end{lem}

\begin{proof}
We need to show that if $Z\left(A\right)\ni x=x_{e}+x_{\perp}$, where
$x_{e}\in A_{e}$ and $x_{\perp}\in\oplus_{e\neq g\in G}A_{g}$, then
$x_{e}\in R$: Indeed for any component $A_{h}$ we have $0=\left[A_{h},x\right]=\left[A_{h},x_{e}\right]+\left[A_{h},x_{\perp}\right]$
and clearly $\left[A_{h},x_{e}\right]$ has degree $h$ whereas the
projection of $\left[A_{h},x_{\perp}\right]$ into $A_{h}$ is $0$.
As a result, $\left[A_{h},x_{e}\right]=0$ and so $x_{e}\in Z\left(A\right)$.
\end{proof}
\begin{cor}
\label{cor:from_center_to_algebra}If $\mathfrak{a}$ is an ideal
of $R$, then $\mathfrak{a}A\cap R=\mathfrak{a}$.
\end{cor}

\begin{proof}
Notice that we got that $R$ is an $R$-direct summand of $A$. So
we may write $A=R\oplus L$. As a result,
\[
\mathfrak{a}A\cap R=\left(\mathfrak{a}\left(R\oplus L\right)\right)\cap R=\left(\mathfrak{a}\oplus\mathfrak{a}L\right)\cap R=\mathfrak{a}.
\]
\end{proof}

We are ready to prove Theorem \ref{idealCorrespondence}.

\begin{proof}[Proof of Theorem \ref{idealCorrespondence}]
Let $I$ be a $G$-graded ideal of $A$ and let $\mathfrak{i}=I\cap R$
be the corresponding ideal of $R$. We consider the graded projection:
\[
p:A/\mathfrak{i}A\to A/I.
\]
 Notice that $p$ is injective when restricted to the $e$-center,
because by Lemma \ref{e_center_preservation}
\[
Z\left(A/\mathfrak{i}A\right)_{e}=R/\mathfrak{i}=Z\left(A/I\right)_{e}.
\]
By the same Lemma, both algebras are Azumaya. We want to show that
$p$ is injective; from which it follows that $I=\left(I\cap R\right)A$.
We are in the following situation: We have a surjection between $p:B\to B^{\prime}$
between two $G$-graded Azumaya algebras, which, when restricted to
the $e$-center, is an isomorphism. Denote by $R_{0}$ the common $e$-center.
By Corollary \corref{maximal_ideals_correspondence} the maximal $G$-graded
ideals of $B$ and $B^{\prime}$ are respectively $\mathfrak{m}B$
and $\mathfrak{m}B^{\prime}$, where $\mathfrak{m}$ is a maximal
ideal of $R_{0}$. Hence, by dividing by $\mathfrak{m}B$ and $\mathfrak{m}B^{\prime}$,
we obtain the graded projection
\[
p_{\mathfrak{m}}:\left(A/\mathfrak{i}A\right)/\mathfrak{m}B\to\left(A/I\right)/\mathfrak{m}B'.
\]
By Lemma \ref{azumaya_over_field} we conclude that $p_{\mathfrak{m}}$
is an isomorphism. As a result, the kernel of $p$ must be contained
in $\bigcap_{\mathfrak{m\text{ is maximal in \ensuremath{R}}}}\mathfrak{m}A$. Since we also assume that $J(R)=0$, we conclude that this intersection is zero.

The other direction, which is, $\mathfrak{i}A\cap R=\mathfrak{i}$
is Corollary \corref{from_center_to_algebra}.
\end{proof}

\section{Generic $G$-graded Azumaya algebras}

Let $A$ be a $G$-graded simple finite dimensional algebra over an algebraically
closed field $F$ of characteristic $0$. Denote by $\mathcal{A}$
its $G$-graded relatively free algebra.

In this section we show that there is an $e$-central localization
$s^{-1}\mathcal{A}$ of $\mathcal{A}$ satisfying all the assumptions
(hence the conclusion) of Theorem \ref{idealCorrespondence}.

Let us state some facts regarding $\mathcal{A}$ which were proved
in \cite{Karasik}:
\begin{prop}
(\cite[Lemmta 5.3,5.4,5.7]{Karasik})The algebra $\mathcal{A}$ has
$J\left(A\right)=0$ and $Z\left(\mathcal{A}\right)_{e}$ has no zero-divisors
of $\mathcal{A}$. Furthermore, if we denote $S=Z\left(\mathcal{A}\right)_{e}\setminus\left\{ 0\right\} $,
then $S^{-1}\mathcal{A}$ is finite dimensional $G$-graded simple algebra
over its $e$-center $R$.
\end{prop}

By the above proposition and \cite[Theorem 2.5]{Braun} we conclude that
$S^{-1}\mathcal{A}$ is separable over $R$. By \cite[Theorem 3.8]{Braun}
$S^{-1}\mathcal{A}$ is Azumaya. Using Braun's characterization of
Azumaya algebras (see \cite{Braun,Dicks}), one can find $s\in S$ and
$a_{1},\dots,a_{m};b_{1},\dots,b_{m}\in\mathcal{A}$ such that
\[
\frac{1}{s}\sum_{i=1}^{m}a_{i}b_{i}=1
\]
and for every $a\in\mathcal{A}$,
\[
\sum_{i=1}^{m}a_{i}ab_{i}\in Z\left(\mathcal{A}\right).
\]
 As a result, $\mathcal{U}:=s^{-1}\mathcal{A}$ is an Azumaya algebra
such that $J(R)=0$.

We can now complete the proof of Theorem \ref{main theorem}. To this end we modify the $G$-graded generic algebra constructed above so that it satisfies in addition the Artin-Procesi condition.
\begin{prop}\label{Azumaya and and Artin Procesi together}
Let $A$ be a finite dimensional $G$-graded simple algebra over an algebraically closed field $F$ of characteristic zero. Let $k$ be its unique minimal field of definition. Let $k\langle X_{G} \rangle/\Gamma$ be the relatively free $G$-graded algebra of $A$. Then there exists an $e$-central polynomial $t(X)$ such that the localized algebra $\mathcal{U} = t(X)^{-1} (k\langle X_{G} \rangle/\Gamma)$ is Azumaya in the sense of Artin-Procesi and satisfies the condition on graded ideals.
\end{prop}
\begin{proof}
Let $t(X) = s\cdot m(X)$. It is clear that the algebra $t(X)^{-1} (k\langle X_{G} \rangle/\Gamma)$ satisfies the conditions of the above paragraph and so it is $G$-graded generic Azumaya. As for the Artin-Procesi property, it is easily seen that if $m(X)$ vanishes on $G$-graded simple algebras $E$ with $Id_{G}(E) \varsupsetneq Id_{G}(A)$, so does $t(X)$ and the result follows.
\end{proof}

We close this section by proving the algebra~ $\mathcal{U} = \mathcal{U}_{A}$ satisfies all conditions stated in Theorem \ref{main theorem}.
\begin{proof}
We need to prove $(5)$ (indeed ($1$) follows from (\cite[Lemmta 5.3,5.4,5.7]{Karasik}), ($2$) follows from the fact that $\mathcal{U}$ is an $e$-central localization of the relatively free algebra $\mathcal{A} = k\langle X_{G} \rangle /\Gamma$ of $A$ and finally ($3$) and ($4$) were taken care in Proposition \ref{Azumaya and and Artin Procesi together}).

Let $\eta: \mathcal{U} \rightarrow B \neq 0$ be a $G$-graded map of $k$ algebras and onto. Let $S$ be the $e$-center of $B$ and suppose it is an integral domain. If $J = ker(\eta)$, we know by Theorem \ref{main theorem} that it is induced by an $e$-central ideal of $R$, the $e$-center of $\mathcal{U}$. It follows that $R/I \cong S$ and $B \cong \mathcal{U}/I\mathcal{U}$ as $G$-graded algebras. Extending scalars to $L = frac (S)$ we obtain a finite dimensional algebra $A_{0}$, $G$-graded simple over its $e$-center $L$. By the Artin-Procesi property of $\mathcal{U}$ we have that $B$ is PI equivalent to $A$ and hence $A_{0}$ is a $G$-graded form of $A$.

Conversely: Let $A_{0}$ be a finite dimensional algebra, $G$-graded simple over a field $L$, the $e$-center of $A_{0}$. Suppose $A_{0}$ is a $G$-graded $L$-form of $A$. Our goal is to exhibit a homomorphic image of $\mathcal{U}$ which yields $A_{0}$ by scalar extension (note that we cannot expect better than that since $\mathcal{U}$ is countable whereas $A_{0}$ is not in general).
Because the polynomial $t(X)$ is $e$-central and nonzero on $A$, it is $e$-central and nonzero on $A_{0}$ and hence there is an evaluation of $t(X)$ on $A_{0}$ with nonzero, and hence invertible value in $L$. This fixes the value of the variables that appear in $t(X)$. We assume as we may that $k\langle X_{G} \rangle$ has additional variables $Y_{G}$, at least as the dimension of $A$ over its $e$-center $F$ (with the corresponding homogeneous degrees). Fix a basis $\beta_{G}$ of $A_{0}$ over $L$ consisting of homogeneous elements and choose an evaluation of the variables $Y_{G}$ onto $\beta_{G}$.  We denote the map obtained from the evaluation by $\phi$ and let $J = ker(\phi)$. Because $J$ is a $G$-graded ideal of $\mathcal{U}$, it corresponds to an ideal $I$ of $R$. The $e$-center of $image(\phi)$ is contained in $L$, and hence is a domain. This shows $I$ is a prime ideal of $R$ and extension of scalars of $image(\phi)$ from $R/I$ to $L$ shows $image(\phi)$ is an $R/I$-form of $A_{0}$.
\end{proof}

\section{$G$-graded verbally prime $T$-ideals and $G$-graded division algebras}

This section is devoted to $G$-graded division algebras and their $T$-ideals of graded identities. The main theorems are Theorem \ref{essentially verbally prime} and Theorem \ref{G = S}.

We start by recalling the characterization of $G$-graded strongly verbally prime $T$-ideals in terms of Bahturin, Sehgal and Zaicev presentation of $G$-graded simple algebras (see Definition \ref{definition verbally and strongly verbally} above and \cite{AljKarasik} for more details).
\begin{thm} \label{stronglyverballyprime}
A $G$-graded $T$-ideal $\Gamma$ over an algebraically closed field $F$ of characteristic zero with $c_{n} \in \Gamma$ the $n$th Capelli polynomial for some $n$, is strongly verbally prime if and only if it is the $T$-ideal of $G$-graded identities of a finite dimensional $G$-graded simple algebra $A$ with a presentation $P_{A} = \{H, \alpha, \mathfrak{g} = (g_{1},\ldots, g_{n})\}$ of the following form $($we assume as we may the $G$-grading is connected, that is, the elements $g\in G$ for which $A_{g} \neq 0$ generate $G$$)$:
\begin{enumerate}
\item
The group $H$ is normal in $G$.
\item
All coset representatives of $H$ in $G$ appear in $\mathfrak{g}$ with equal frequency.
\item
The class $\alpha \in H^{2}(H, F^{*})$ is $G$ invariant where the action of $G$ on $F^{*}$ is trivial.
\end{enumerate}
\end{thm}
As was mentioned in the introduction these are precisely the $G$-graded simple algebras $A$ which admit a $G$-graded division algebra for over a field $k$ which contains an algebraically closed field $F$.

In the next theorem we characterise $G$-graded simple algebras $A$ which admit a $G$-graded division algebra in terms of their graded presentation $P_{A}$. The terminology is as in Theorem \ref{stronglyverballyprime} (in particular we assume the $G$-grading on $A$ is connected).

\begin{thm} \label{G = S}
Let $A$ be a finite dimensional $G$-graded simple algebra over $F$. Then it has a $G$-graded division algebra form over its $e$-center if and only if $G = S$ $($that is $H$ is normal in $G$, $H$-cosets are equally represented in $\mathfrak{g}$ and $G$ normalizes $B_{\alpha}$$)$.
\end{thm}

Before we proceed with the proofs let us introduce the following terminology.  If $K$ is a subfield of $F$, we say $p \in F\langle X_{G} \rangle$ is a $K$-polynomial if its coefficients are in $K$.
\begin{defn}

Let $\Gamma$ be a $G$-graded $T$-ideal over an algebraically closed field $F$ of characteristic zero and let $K$ be any subfield. We say that $\Gamma$ is \textit{$K$-strongly verbally prime} if $\Gamma$ has no nonzero $G$-homogeneous zero divisors over $K$, that is for any two multilinear $G$-homogeneous $K$-polynomials $p$ and $q$, on disjoint sets of variables, their product $pq$ is in $\Gamma$ if and only if $p$ or $q$ is in $\Gamma$.
\end{defn}

\begin{rem}
Note that with this terminology

\begin{enumerate}
\item
A $T$-ideal $\Gamma$ is \textit{essentially verbally prime} (Definition \ref{essentially verbally prime definition}) if and only if it is $k$-strongly verbally prime where $k$-is its unique minimal field of definition.
\item
A $G$-graded $T$-ideal $\Gamma$ being strongly verbally prime in the sense of Theorem \ref{stronglyverballyprime} if and only if it is $F$-strongly verbally prime where $F$ is an algebraically closed field.
\end{enumerate}
\end{rem}

\begin{proof} {of Theorems \ref{G = S} and \ref{essentially verbally prime}}

We will show the following conditions are equivalent.

\begin{enumerate}
\item
$A$ admits a $G$-graded division algebra form over its $e$-center.
\item
$Id_{G}(A)$ is essentially verbally prime.
\item
$Id_{G}(A)$ is $\mathbb{Q}$-strongly verbally prime.
\item
$G = S$.
\end{enumerate}

$1\Rightarrow 2:$ Let $B$ be a $G$-graded division algebra form of $A$ over its $e$-center $K$, that is there is a field $L$ extending $K$ and $F$ such that $B\otimes_{K}L \cong A\otimes_{F}L$ as $G$-graded algebras. If $k$ is the minimal field of definition of $A$, we have $K \geq k$ and so, if $p$ and $q$ are multilinear homogeneous $k$-polynomials, nonidentities of $A$ with disjoint sets of variables, we have that $p$ and $q$ are nonidentities also of $B$ and so nonzero evaluations on $B$ must be invertible. This implies their product is nonzero and in particular $pq$ is a nonidentity of $B$ and hence of $A$.

$2\Rightarrow 3:$ Clear.

$3\Rightarrow 4:$ We need to show if $P_{A} = \{H, \alpha, \mathfrak{g} = (g_{1},\ldots, g_{n})\}$ is a presentation of $A$ then $H$ is normal, all cosets of $H$ in $G$ appear with equal frequency in the tuple $\mathfrak{g}$ and finally the subgroup $B_{\alpha}$ of $M(H)$ is normalized by $G$. Now, following the proof of Theorem \ref{stronglyverballyprime} in \cite{AljKarasik} we have that if $H$ is not normal in $G$ or the cosets of $H$ in $G$ do not appear with equal frequency in $\mathfrak{g}$ one can construct $\mathbb{Q}$-polynomials, multilinear, $G$-homogeneous nonidentities with disjoint sets of variables whose product is an identity. This contradicts $(3)$.

In order to show $G$ normalizes $B_{\alpha}$ (i.e. $G = S$) we may assume $H$ is normal in $G$ and the $H$-cosets are represented in $\mathfrak{g}$ with equal frequency. Note that in that case we have a well defined action of $G$ on $M(H)$.

By contradiction suppose $g(B_{\alpha}) \nsubseteq B_{\alpha}$ for some $g\in G$ and let $z \in B_{\alpha}$ be such that $g(z) \notin B_{\alpha}$. Applying Proposition \ref{binomial identities and Schur multiplier} there is a monomial $x_{1, h_1}\cdots x_{r, h_r}$, $h_{i} \in H$, and element $\sigma \in Sym(r)$, such that $$z = x_{1, h_1}\cdots x_{r, h_r}\cdot (x_{\sigma(1), h_\sigma(1)}\cdots x_{\sigma(r), h_\sigma(r)})^{-1} \in B_{\alpha}$$ but $$z^{g} = x_{1, h_1}^{g}\cdots x_{r, h_r}^{g}\cdot (x_{\sigma(1), h_\sigma(1)}^{g}\cdots x_{\sigma(r), h_\sigma(r)}^{g})^{-1} \notin B_{\alpha}.$$ It follows that
$$
\beta_{\sigma}(x_{1,h_1},\ldots, x_{r, h_{r}}) = x_{1, h_1}\cdots x_{r, h_r} - x_{\sigma(1), h_\sigma(1)}\cdots x_{\sigma(r), h_\sigma(r)}
$$
is a $H$-graded identity of $F^{\alpha}H$ whereas
$$
x_{1, h_1}^{g}\cdots x_{r, h_r}^{g} - x_{\sigma(1), h_\sigma(1)}^{g}\cdots x_{\sigma(r), h_\sigma(r)}^{g}
$$
is not.

Now, we order the tuple $\mathfrak{g} = (g_{1}=e, g_{2}, \ldots, g_{s})$ as follows. For simplicity it starts with the representative $e$ followed with all coset representatives $g_{i}$ with $g_{i}(z) \in B_{\alpha}$. Then, on the right of the tuple, we put the remaining elements $g_{i}$, namely those with $g_{i}(z) \notin B_{\alpha}$. We arrange equal representatives to be adjacent in $\mathfrak{g}$. Note that in each category, namely coset representatives that leave $z$ in $B_{\alpha}$ and those that do not, there is at least one element, that is $e$ and $g$ respectively.

Now, we construct the Capelli type polynomial, nonidentity, which alternates on the $e$-blocks. We insert the polynomial
$$
\beta_{\sigma, d} = \beta_{\sigma}(R_{d, h_{1}},\ldots, R_{d, h_{r}})
$$
on the right of each segment which corresponds to one of the blocks on the right part of the tuple $\mathfrak{g}$, namely the blocks that correspond to coset representatives $g_{j}$ with $g_{j}(z) \notin B_{\alpha}$. Here $R_{d, h_{i}}$ is the polynomial obtained from Regev's central polynomial for $d \times d$-matrices where all its variables but one are replaced by different $e$-variables and the remaining variable is replaced by an $h_{i} \in H$ variable (see Section $2$ above).

This implies that if we evaluate one of these segments on a block represented by $g_{i}(z) \in B_{\alpha}$, we get zero. So for a nonzero evaluation we are forced to evaluate the segments on the left with the first blocks. This shows the product of several (in fact two) such polynomials, different variables is an identity and so we have nontrivial zero divisors over $\mathbb{Q}$. Again this contradicts $(3)$ and so we are done with the proof of $3\Rightarrow 4$.

$4\Rightarrow 1:$ Assume $4$. Let $\mathcal{A}$ be the relatively free algebra of $A$ over $k$. We claim it is sufficient to show $\mathcal{A}_{e}$ is a domain: Recall the $e$-center of $\mathcal{A}$ has no nontrivial zero divisors in $\mathcal{A}$ and moreover, extending the $e$-center to its field of fractions $E$ yields a $G$-graded simple algebra $\Lambda$ finite dimensional over its $e$-center $E$. But being the $e$-component of $\Lambda$ a domain and finite dimensional over $E$ we have $\Lambda_{e}$ is a division algebra and hence $\Lambda$ is a $G$-graded division algebra. This proves the claim. We refer to the algebra $\Lambda$ as the \textit{generic $G$-graded division algebra} corresponding to $A$ over its field of definition.

In order to show $\mathcal{A}_{e}$ is a domain take $p$ and $q$ $k$-polynomials, $e$-homogeneous   nonidentities of $A$. We need to prove $pq$ is a nonidentity. For this we make some reductions following the approach in \cite{AljKarasik}.

Recall $p \in k\left\langle X_{G}\right\rangle$ is $G$-\textit{homogeneous} if it belongs to $k\left\langle X_{G}\right\rangle_{g}$ for some $g \in G$. We say $p$ is \textit{multihomogeneous } if any variable $x_{i_g}$ appears the same number of times in each monomial of $p$. More restrictive, we recall that $p$ is \textit{multilinear} if it is multihomogeneous and any variable $x_{i_g}$ appears at most once in each monomial of $p$. Note that unless $G$ is abelian, multihomogeneous or multilinear polynomials are not necessarily $G$-homogeneous.

We first reduce the problem to polynomials $p$ and $q$, $e$-homogeneous and multihomogeneous.

To this end we introduce an ordering on the variables which appear either in $p$ or $q$. Then counting variables multiplicities we order lexicographically the monomials in $p$ and $q$ respectively. Let $p=p_1+\ldots +p_r$ and $q=q_1+\dots +q_s$ be the decomposition of $p$ and $q$ into its multihomogeneous constituents where $p_1 < \ldots < p_r$ and $q_1 < \ldots < q_s$. Suppose $pq$ is a $G$-graded identity. It follows that its homogeneous constituent are $G$-graded identities and in particular $p_1q_1$ is a $G$-graded identity. This proves the reduction to multihomogeneous polynomials.

Let $p$ be an $e$-homogeneous, multihomogeneous with coefficients in $k$. Suppose $p$ is a nonidentity of $A$. It is sufficient to prove the following proposition.

\begin{prop}\label{evaluations on the algebra} Evaluations of $p$ on the algebra $A$ yield nonzero values on \textit{every} component of $A_{e}$. More precisely, if we denote by $P_{i}$ the projection $A_{e} \rightarrow A_{e,i}$, then for any $i$ the evaluation set of $P_{i}p$ on $A$ is nonzero.
\end{prop}

Let us complete the proof of $4\Rightarrow 1$ using the proposition. We let $\mathfrak{A}_{p,i} = \{z \in A^{n}: P_{i}p(z) \neq 0\}$  and  $\mathfrak{B}_{q,i} = \{z \in A^{n}: P_{i}q(z) = 0\}$. As $\mathfrak{A}_{p,i}$ is nonempty and open, $\mathfrak{B}_{q,i} \neq A^{n}$ and closed, we have $\mathfrak{A}_{p,i} \nsubseteq \mathfrak{B}_{q,i}$. This shows $P_{i}pq=P_ip\cdot P_iq: A^{n} \rightarrow M_{r}(F)$ is a nonzero polynomial function and we are done.

Note that by the reduction above we need to prove the proposition for multihomogeneous, $e$-homogeneous $k$-polynomials.

Claim: \textit{It is sufficient to prove the proposition for $p$ multilinear.}

We assume the proposition holds for multilinear nonidentities and suppose by contradiction $p(x_{1},x_{2}\ldots,x_{n})$ is a nonidentitiy, multihomogeneous, nonmultilinear for which $P_{i}p(x_1,x_2,\ldots,x_n)=0$ for some $i$.
Suppose the variables $x_{1},\ldots,x_{n}$ appear in $p$ with degrees $k_1,\ldots,k_n$. Without loss of generality we may assume $k_1 \geq k_2 \geq \cdots \geq k_n$, and hence $k_{1} > 1$, making the tuple $(k_1,\ldots,k_n)$ a partition of $m=\sum_{s}k_{s}$. Clearly, we may assume the polynomial $p$ is a counter example of minimal degree in the lexicographic ordering determined by the degrees of $x_{1},\ldots,x_{n}$.
Put $x_{1}=t_{1}+\cdots +t_{k_1}$ and let
$$\bar{p}(t_1,\ldots,t_{k_1},x_{2},\ldots,x_{n}) = p((t_1+\ldots + t_{k_1}),x_{2},\ldots,x_{n}).$$
We have that $P_{i}\bar{p}(t_1,\ldots,t_{k},x_{2},\ldots,x_{n})=0$. Furthermore, replacing $x_{1}$ by $t_{j}$ we have also $P_{i}p(t_{j},x_{2},\ldots,x_{n})=0$, for $j=1,\ldots,k_1$, and so $P_{i}(q) = 0$ where

$$q(t_1,\ldots,t_{k_1},x_{2},\ldots,x_{n}) = \bar{p}(t_1,\ldots,t_{k_1},x_{2},\ldots,x_{n}) - \sum_{j=1}^{k_1}p(t_{j},x_{2},\ldots,x_{n}))$$
is a sum of multihomogeneous polynomials of degree $m$ whose multihomogeneous degrees in the lexicographic ordering are strictly smaller than $(k_1,\ldots,k_n)$. Finally, note that the polynomial $p(x_{1},x_{2}\ldots,x_{n})$ is obtained is in the $T$-ideal generated by $q$ and hence the latter is a nonidentity of $A$.

Next we proceed to prove the proposition for multilinear polynomials, first with the extra condition that each $H$ coset is represented only once in $\mathfrak{g}$, that is $d = 1$. To this end we recall here the notions of \textit{good permutations} on multilinear monomials, \textit{pure polynomials} and the \textit{path property} as they were presented in \cite{AljKarasik}. Recall we are assuming $A$ is $G$-graded simple with presentation $P_{A} = \{H, \alpha, \mathfrak{g} = (g_{1},\ldots, g_{n})\}$ and $G = S$. In particular $H$ is normal in $G$. For $g \in G$ denote by $\bar{g}$ its image in $G/H$.
\begin{defn} (see \cite{AljKarasik} Definition $4.13$) \label{Good and Pure}

We say the $G-$graded
multilinear monomial $Z_{\sigma}=x_{\sigma(1),t_{\sigma(1)}}\cdots x_{\sigma(n),t_{\sigma(n)}}$ is a \textit{good permutation}
of $Z=x_{1,t_{1}}\cdots x_{n,t_{n}}$
if
\begin{enumerate}
\item $(\mbox{deg}Z=)t_{1}\cdots t_{n}=t_{\sigma(1)}\cdots t_{\sigma(n)}(=\mbox{deg}Z_{\sigma})$.
\item For every $1\leq i\leq n$, $\overline{t_{1}\cdots t_{i}}=\overline{t_{\sigma(1)}\cdots t_{\sigma(\sigma^{-1}(i))=i}}$.
\end{enumerate}

A $G$-graded multilinear polynomial is said to be \textit{pure} if all its monomials are \textit{good permutations} of each other.
\end{defn}
For the definition of the path property we refer to \cite{AljKarasik} Definition $4.16$ where here, in addition, we assume $d = 1$ (each coset representative appears only once in $\mathfrak{g}$).

Let $p(x_{t_1},\ldots,x_{t_n})$ be a $G$-graded \textit{pure} polynomial.

\begin{defn}

\begin{enumerate}
\item

We say the polynomial $p(x_{t_1},\ldots,x_{t_n})$ satisfies the \textit{path property} if the following condition holds: If $p$ vanishes whenever the evaluation of $x_{t_1}$ has the form $u_{h(i_1)}\otimes e_{i_1,j(i_1)} $ for some index $i_1 \in \{1,\ldots,[G:H]\}$, then $p$ is a $G$-graded identity of $A$.

\item
We say the $G$-graded simple algebra $A$ satisfies the \textit{path property} if any \textit{pure} polynomial satisfies the \textit{path property}.

\end{enumerate}
\end{defn}

The following lemma extends \cite{AljKarasik} Lemma $4.17$.
\begin{lem}
Let $A$ be a $G$-graded simple algebra with presentation $P_{A}$. Suppose $G = S$ and let $k$ be the unique minimal field of definition of $A$. Suppose that the frequency of the representatives of $H$ cosets in $\mathfrak{g}$ is $1$. Then $A$ satisfies the path condition on pure $k$-polynomials, that is, if $p = \sum_{\sigma \in Sym(n)} b_{\sigma} Z_{\sigma}$ is a pure $k$-polynomial, then it vanishes on one path if and only if $p$ is a graded identity of $A$.
\end{lem}
\begin{proof} The proof here is an adaptation of the proof in \cite{AljKarasik}. We provide some details.
Let $Z$ be a multilinear monomial of degree $n$ and $Z_{\tau}$ is the monomial obtained from $Z$ by permuting the variables with some $\tau \in Sym(n)$. Choose a nonzero (path) evaluation of $Z$ on basis elements and suppose the value is in the $(i,i)$-block. Then the following hold (see proof of Lemma $4.15$ in \cite{AljKarasik}).

\begin{enumerate}
\item
If $Z_{\tau}$ is a good permutation of $Z$, then the value of $Z_{\tau}$ is nonzero and lies in the $(i,i)$-block.
\item
If $Z_{\tau}$ is not a good permutation of $Z$, then the value of $Z_{\tau}$ lies in the $(j,j)$-block for some $j \neq i$ or is zero.

\end{enumerate}
Next suppose $Z$ and $Z_{\sigma}$ are good permutations of each other and let $\hat{Z}_{1} = u_{e}\otimes e_{1,1}$ denote the evaluation $Z$ on the $1$st path. By the preceding statement, we let $(\hat{Z}_{\sigma})_{1} = \alpha\hat{Z}_{1} = \alpha\cdot u_{e}\otimes e_{1,1}$ be the corresponding evaluation where $\alpha = \alpha_{\sigma}$ is an $n$th root of unity. Let us evaluate the monomials $Z$ and $Z_{\sigma}$ on the $i$th path. Following the proof of (\cite{AljKarasik}, Proposition $4.20$) verbatim we obtain that $(\hat{Z}_{\sigma})_{i} = g_{i}(\alpha_{\sigma})\hat{Z}_{i}$.

We now show that if $p = \sum_{\sigma}b_{\sigma}Z_{\sigma}$ is a pure polynomial with coefficients in $k$ then it satisfies the path property on $A$.
By symmetry we may assume it vanishes on the $1$st path so $\sum_{\sigma}b_{\sigma}\alpha_{\sigma} = 0$. The value on the $i$th path is $\sum_{\sigma}b_{\sigma}g_{i}(\alpha_{\sigma})\hat{Z}_{i}$. But $b_{\sigma} \in k$, every $\sigma$, and so
$$\sum_{\sigma}b_{\sigma}g_{i}(\alpha_{\sigma})\hat{Z}_{i} = g_{i}(\sum_{\sigma}b_{\sigma}\alpha_{\sigma})\hat{Z}_{i} = 0.$$
\end{proof}

We complete now the proof of Proposition \ref{evaluations on the algebra}. Firstly we prove it for pure $k$-polynomials and $e$-homogeneous, and then later we extend the proof to multilinear $k$-polynomials and $e$-homogeneous.

Suppose $p$ is $e$-homogeneous and pure. We know that since $p$ is a $k$-polynomial, if it does not vanish on a certain path, say the $i$th path, it does not vanish on any path $j \in \{1,\ldots, [G:H]\}$. Furthermore, we saw that different paths yield values on different diagonal blocks and so the claim is clear in this case.

Suppose now $p$ is multilinear and let $p = p_1 + \ldots + p_q$ be the decomposition of $p$ into its pure components. Let $Z$ be a monomial of $p_1$ say and suppose it does not vanish on a certain path evaluation. Let $Z_{\sigma}$ be a monomial in $p$ which is obtained from $Z$ by the permutation $\sigma \in S_{n}$. If $Z_{\sigma}$ is a good permutation of $Z$ then they are constituents of the same pure polynomial, the case we considered above. Otherwise, applying the claim in the proof of Lemma 4.16 in \cite{AljKarasik} we know $Z$ and $Z_{\sigma}$ get values in different blocks (this includes the case the evaluation annihilates $Z_{\sigma}$). This establishes the claim in case $p$ is multilinear.

We can now complete the proofs of Theorems \ref{G = S} and \ref{essentially verbally prime}.
\end{proof}

\begin{proof}
We have completed the proofs of $1\Rightarrow 2 \Rightarrow 3 \Rightarrow 4$ whereas the implication $4\Rightarrow 1$ was shown under the condition $d=1$. Let us prove $4\Rightarrow 1$, $d > 1$, assuming the case $d=1$. Write the $G$-graded simple algebra $A$ as the tensor product of $A_{1} \otimes M_{d}(F)$ where the grading on $M_{d}(F)$ is trivial. Applying the case $d=1$ we know $A_{1}$ has a $G$-graded division algebra form $B_{1}$. In order to complete the proof we need to find a division algebra form for the algebra of $d\times d$- matrices such that tensoring with $(B_{1})_{e}$ yields a division algebra. This can be done by taking a generic division algebra of index $d$ whose variables are disjoint to the variables appearing in $(B_{1})_{e}$. Details are omitted.
This completes the proof of implication $4\Rightarrow 1$ and therefore the proofs of Theorems \ref{G = S} and \ref{essentially verbally prime} are now complete.

\end{proof}

Final remarks: It is natural to ask whether some of the statements presented in this article may be extended to $G$-graded simple algebras over fields of characteristic $p > 0$. Let us make some comments.

Bahturin, Sehgal and Zaicev' characterization of $G$-graded simple algebras requires a Maschke's type condition, namely that $p$ does not divide the order of $G$ (\cite{BSZ}). In fact it would be sufficient to assume $p$ does not divide the order of $H$ where as above, $H$ is the group appearing in the presentation $P_{A}= \{H, \alpha, \mathfrak{g} = (g_{1},\ldots,g_{n})\}$. We note that it would be necessary to make such an assumption since a twisted group algebra $F^{\alpha}H$, $F$ algebraically closed (or more generally, if $F$ is perfect) of characteristic $p$ is semsimple if and only if the order of $H$ is prime to $p$ (see \cite{AljRob}, Theorem $2$).
Another point which should be addressed is the extension (or possibly \textit{replacement}) of Kemer's representability theorem or what is also referred in the literature as \textit{Kemer's finite dimensionality theorem} for $G$-graded algebras over fields of characteristic zero (every affine PI algebra and $G$-graded is PI equivalent to a finite dimensional algebra $G$-graded algebra  \cite{AB}) to $G$-graded algebras over fields of positive characteristic. For instance we apply Kemer's finite dimensionality theorem to prove Part $3$ of Theorem \ref{main theorem}, that is, to show the algebra $\mathcal{U}$ is Azumaya in the sense of Artin-Procesi. It is known that Kemer's finite dimensionality theorem holds for ungraded affine PI algebras over infinite fields however it is false for algebras over finite fields. In the latter case we have that an affine relatively free algebra can be embedded into a Noetherian module over its centroid. For the treatment of representability and Specht problem for (ungraded) algebras over fields of positive characteristic we refer the reader to articles of Belov and Belov, Rowen and Vishne (see \cite{Belov}, \cite{BelovRoweVishne1}, \cite{BelovRoweVishne2}, \cite{BelovRoweVishne3}).


\begin{thebibliography}{10}

\bibitem{AljHaile} E. Aljadeff and D. Haile, {\em Simple $G$-graded algebras and their polynomial identities}, Trans. Amer. Math. Soc. 366 (2014), 1749--1771.

\bibitem{AljHaileNat} E. Aljadeff, D. Haile, Darrell and M. Natapov, {\em Graded identities of matrix algebras and the universal graded algebra}, Trans. Amer. Math. Soc. 362 (2010), no. 6, 3125--3147.

\bibitem{AB} E. Aljadeff, A. Kanel-Belov, {\em Representability and Specht problem for $G-$graded algebras}, Adv. Math. {\bf 225} (2010), no. 5, 2391--2428.

\bibitem{AlBeKar} E. Aljadeff, A. Kanel-Belov and Y. Karasik {\em Kemer's theorem for affine PI algebras over a field of characteristic zero}, J. Pure Appl. Algebra 220 (2016), no. 8, 2771--2808.

\bibitem{AljKarasik} E. Aljadeff and Y. Karasik, {\em Verbally prime T-ideals and graded division algebras}, Adv. Math. 332 (2018), 142--175.


\bibitem{AGPR} E. Aljadeff, A. Giambruno, C. Procesi and A. Regev, {\em Rings with polynomial identities and finite dimensional representations of algebras}, American Mathematical  Society Colloquium Publications, vol. $66$, American Mathematical Society, [Providence], RI, [2020], $©$ 2020.

\bibitem{AljRob} E. Aljadeff and D.J.S. Robinson, {\em Semisimple algebras, Galois actions and group cohomology}, J. Pure Appl. Algebra 94 (1994), no. 1, 1-15.

\bibitem{Amitsur noncrossed products} Amitsur, S. A. {\em The generic division rings}, Israel J. Math. 17 (1974), 241--247.

\bibitem{Vishne etal problems} A. Auel, E. Brussel, S. Garibaldi and U. Vishne {\em Open problems on central simple algebras}, (English summary) Transform. Groups 16 (2011), no. 1, 219--264.

\bibitem{Braun} Braun, Amiram, {\em On Artin's theorem and Azumaya algebras}, Journal of Algebra 77.2 (1982), 323-332.

\bibitem{BEK} Y. Bahturin, A. Elduque and M. Kochetov {\em Graded-division algebras over arbitrary fields}, J. Algebra Appl. 20 (2021), no. 1, Paper No. 2140009, 28 pp.

\bibitem{BahturinZaicev1} Y. Bahturin and M. Zaicev, {\em Simple graded division algebras over the field of real numbers}, Linear Algebra Appl. 490 (2016), 102-123.

\bibitem{BahturinZaicev2} Y. Bahturin and M. Zaicev, {\em Graded division algebras over the field of real numbers}, J. Algebra 514 (2018), 273-309.

\bibitem{BSZ} Y. Bahturin, M. Zaicev, and S. K. Sehgal, {\em Finite-dimensional simple graded algebras}. (Russian), Mat. Sb., 199(7), (2008), 21--40; translation, Sb. Math., 199(7), (2008), 965--983.

\bibitem{BalbaMikhalev} I. N. Balaba and A. V. Mikhalev, {\em Graded division rings} Sarajevo J. Math. 14(27) (2018), no. 2, 167-174.

\bibitem{Belov}, A. Kanel-Belov, {\em Local finite basis property and local representability of varieties of associative
rings}, (Russian. Russian summary)
Izv. Ross. Akad. Nauk Ser. Mat. 74 (2010), no. 1, 3–134; translation in Izv. Math. 74
(2010), no. 1, 1-126.

\bibitem{BelovRoweVishne1} A. Kanel-Belov, L.H. Rowen and U.Vishne, {\em Full quivers of representations of
algebras}, Trans. Amer. Math. Soc. 364 (2012), no. 10, 5525-5569.

\bibitem{BelovRoweVishne2} A. Kanel-Belov, L.H. Rowen and U.Vishne, {\em PI-varieties associated to full quivers
of representations of algebras}, Trans. Amer. Math. Soc. 365 (2013), no. 5, 2681-2722.

\bibitem{BelovRoweVishne3} A. Kanel-Belov, L.H. Rowen and U.Vishne, {\em Specht's problem for associative affine algebras over commutative Noetherian rings}, Trans. Amer. Math. Soc. 367 (2015), no. 8, 5553-5596.

\bibitem{BereleRowen} A. Berele and L. Rowen, {\em $T$-ideals and super-Azumaya algebras}, J. Algebra 212 (1999), no. 2, 703-720.

\bibitem{Brussel} E. Brussel, {\em Noncrossed products and nonabelian crossed products over $\mathbb{Q}(t)$ and $\mathbb{Q}((t))$}. Amer. J. Math. 117 (1995), no. 2, 377--393.

\bibitem{Childs} L. N. Childs, G. Garfinkel and M. Orzech, {\em The Brauer group of graded Azumaya algebras} Trans. Amer. Math. Soc. 175 (1973), 299-326.

\bibitem{OfirDavid} O. David, {\em Graded embeddings of finite dimensional simple graded algebras}, J. Algebra 367 (2012), 120--141.

\bibitem{Dicks} W. Dicks, {\em On a characterization of Azumaya algebras}, Publicacions Matematiques 32.2 (1988), 165--166.

\bibitem{ElduqueKochetov} A. Elduque and M. Kochetov, {\em Gradings on simple Lie algebras}, Mathematical Surveys and Monographs, vol. 189, American Mathematical Society, Providence, RI; Atlantic Association for Research in the Mathematical Sciences (AARMS), Halifax, NS, 2013. MR3087174.

\bibitem{ElduqueKotchetov1} A. Elduque and M. Kochetov, {\em Graded-division algebras and Galois extensions}, J. Pure Appl. Algebra 225 (2021), no. 12, Paper No. 106773, 34 pp.

\bibitem{GilleSzamuely} P. Gille and T. Szamuely, {\em Central simple algebras and Galois cohomology}, Cambridge Studies in Advanced Mathematics, 101. Cambridge University Press, Cambridge, 2006.

\bibitem{Jacobson} N. Jacobson, {\em Finite-dimensional division algebras over fields}, Springer-Verlag, Berlin, 1996.

\bibitem{Karasik} Y. Karasik {\em $G$-graded central polynomials and $G$-graded Posner's theorem}, Transactions of the American Mathematical Society 372.8 (2019): 5531--5546.

\bibitem{Karrer}  G. Karrer, {\em Graded division algebras} Math. Z. 133 (1973), 67-73.

\bibitem{Kemer1984} A. R. Kemer {\em Varieties and $Z_2$-graded algebras}. (Russian) Izv. Akad. Nauk SSSR Ser. Mat. 48 (1984), no. 5, 1042--1059.

\bibitem{Long} F. W. Long, {\em A generalization of the Brauer group of graded algebras} Proc. London Math. Soc. (3) 29 (1974), 237-256.

\bibitem{EssDim} M. Lorenz, Z. Reichstein, L. H. Rowen and D. J. Saltman, {\em Fields of definition for division algebras}, J. London Math. Soc. (2) 68 (2003), no. 3, 651--670.

\bibitem{Ehud} E. Meir, {\em Descent, fields of invariants, and generic forms via symmetric monoidal categories}, J. Pure Appl. Algebra 220 (2016), no. 6, 2077--2111.

\bibitem{MerSus} A. S. Merkurjev, Alexander and A. A. Suslin, {\em K-cohomology of Severi-Brauer varieties and the norm residue homomorphism}. (Russian)
Izv. Akad. Nauk SSSR Ser. Mat. 46 (1982), no. 5, 1011--1046, 1135--1136.

\bibitem{SERRE} J.-P. SERRE, {\em Corps locaux}, Hermann, Paris, 1968.

\bibitem{Small}	C. Small, {\em The Brauer-Wall group of a commutative ring}, Trans. AMS 156 (1971), 455-491.

\bibitem{WALL} C. T. C. Wall, {\em Graded Brauer groups}, J. Reine Angew. Math. 213 (1963/64), 187-199.


\end{thebibliography}
\end{document}